\documentclass[11pt]{article} 

\usepackage[utf8]{inputenc} 

\usepackage{cancel}

\usepackage{geometry} 
\usepackage{graphicx} 
\usepackage[english]{babel}

\usepackage{booktabs} 
\usepackage{array} 
\usepackage{paralist} 
\usepackage{verbatim} 
\usepackage{subfig} 
\usepackage{amsfonts}
\usepackage{amssymb}
\usepackage{amsmath}
\usepackage{amsthm}
\usepackage{makecell}

\usepackage{epigraph}
\usepackage{bbm}

\usepackage{hyperref}
\hypersetup{
    colorlinks=true,
    citecolor=blue
}

\usepackage{bbm}

\renewcommand{\le}{\leqslant}
\renewcommand{\leq}{\leqslant}
\renewcommand{\geq}{\geqslant}
\renewcommand{\ge}{\geqslant}

\newcommand{\id}{\mathbbmss 1}

\newcommand {\matl}{\left[ \begin{matrix}}
\newcommand {\matr}{\end{matrix}\right]}
\newcommand {\Exp}{ \mathbb E }

\renewcommand {\Pr}{ \mathbb P }

\newcommand{\CI}{\operatorname{CI}}

\newcommand{\e}{\mathrm e}

\renewcommand{\d}{\mathrm{d}}

\newcommand{\cP}{\mathcal{P}}

\newcommand{\iid}{\overset{\mathrm{iid}}{\sim}}

\newcommand{\avgX}[1]{\widehat{\mu}_{#1}}
\newcommand{\avgXsq}[1]{\overline{X^2_{#1}}}

\newcommand{\normal}[2]{N_{#1,#2}}

\DeclareMathAlphabet{\mathbbmsl}{U}{bbm}{m}{sl}

\DeclareMathOperator*{\Prw}{\Pr}

\DeclareMathOperator*{\erf}{erf}

\newcommand{\eR}{\overline{\mathbb R^+}}

\newcommand{\subgaussian}[2]{\mathcal{G}_{#1,#2}}
\newcommand{\subgaussianproc}[2]{\mathfrak{G}_{#1,#2}}

\newcommand{\px}{\mathsf{xb}}

\newcommand{\rev}[1]{#1}

\usepackage[round]{natbib}
\usepackage{wrapfig}
\usepackage{subfig}
\usepackage[capitalize,noabbrev]{cleveref}
\usepackage{authblk}
\usepackage{changepage}

\title{The extended Ville's inequality for \\  nonintegrable nonnegative supermartingales\footnote{to appear in \emph{Bernoulli}}}
\author[1]{Hongjian Wang}
\author[1,2]{Aaditya Ramdas}
\affil[1]{Department of Statistics and Data Science, Carnegie Mellon University}
\affil[2]{Machine Learning Department, Carnegie Mellon University} 
\affil[ ]{\texttt{ \{hjnwang,aramdas\}@cmu.edu  }}

\date{\today}

\newtheorem{theorem}{Theorem}[section]
\newtheorem{definition}[theorem]{Definition}

\newtheorem{proposition}[theorem]{Proposition}

\newtheorem{corollary}[theorem]{Corollary}
\newtheorem{lemma}[theorem]{Lemma}
\newtheorem{remark}[theorem]{Remark}
\newtheorem{example}[theorem]{Example}

\usepackage{tocbasic}
\DeclareTOCStyleEntry[
  beforeskip=.05em plus .6pt,
  pagenumberformat=\textbf
]{tocline}{section}

\begin{document}

\maketitle

\begin{abstract}
Following the initial work by Robbins, we rigorously present an extended theory of nonnegative supermartingales, requiring neither integrability nor finiteness. In particular, we derive a key maximal inequality foreshadowed by Robbins, which we call the extended Ville's inequality, that strengthens the classical Ville's inequality (for integrable nonnegative supermartingales), and also applies to our nonintegrable setting. We derive an extension of the method of mixtures, which applies to $\sigma$-finite mixtures of our extended nonnegative supermartingales. We present some implications of our theory for sequential statistics, such as the use of \emph{improper} mixtures (priors) in deriving nonparametric confidence sequences and (extended) e-processes.

\end{abstract}

{
    
    \setcounter{tocdepth}{1}
    \hypersetup{linkcolor=black}
    \tableofcontents
}

 \newpage

\section{Introduction}

The textbook theory of martingales is based centrally on conditional expectations, which in turns is typically formulated for \emph{integrable} random variables, viz., $X$ such that $\Exp[|X|] < \infty$. Thus, when we speak of a martingale, or a supermartingale, it is by default understood to be an $L^1$ process. 
\rev{However, it is less known that} 
for \emph{nonnegative} (super)martingales, integrability can in fact be dropped \citep[Chapter 5.2]{stroock2010probability}.
\rev{In this work, we first introduce from scratch the self-contained theory of} nonintegrable nonnegative martingales, supermartingales, and e-processes (the latter being an object of recent fascination in the literature, that we will define later on). \rev{A topic visited by very few authors, this theory generalizes the usual $L^1$ theory of these classes of processes.}
We \rev{then} describe the utility of such nonintegrable nonnegative martingales in parametric and nonparametric statistics: it allows a frequentist handling of improper priors, a common (but controversial) idea in the Bayesian literature.

As background, the use of nonnegative martingales (NMs) and nonnegative supermartingales (NSMs) dates back at least to 
the doctoral thesis by~\cite{ville1939etude} involving wealths of gambling strategies, as well as Wald's sequential test based on likelihood ratios \citep{wald1945sequential}: Wald did not explicitly use martingales, but Ville did. 

Gambler wealths and likelihood ratios are of course the prototypical example of NMs.
Since Wald's work, NMs and NSMs have been central objects in sequential statistics~\citep{robbins1970statistical,lai1976confidence},  game-theoretic probability~\citep{shafer2005probability,shafer2019game} and game-theoretic statistics~\citep{ramdas2022game}. 
A large part of this progress stems from the central \emph{Ville's inequality}, which states that for any NSM $M$,
\begin{equation}\label{eqn:vil}
     \Pr[ \exists n \ge 0, \ M_n \ge \Exp[M_0]/\varepsilon   ] \le \varepsilon.
\end{equation}
Ville's inequality \eqref{eqn:vil} has been used to design tight self-normalized concentration inequalities \citep{howard2020time}, power-one sequential tests~\citep{darling1968some}, and confidence sequences~\citep{darling1967confidence}. A confidence sequence is a sequence of confidence intervals that are valid at arbitrary stopping times.
\cite{howard2021time} showed how to construct \emph{nonparametric} confidence sequences by combining composite NSMs (processes that are NSMs under a family of distributions), along with the ``method of mixtures'' and Ville's inequality. See also \cite{waudby2020estimating} for a historical review and recent advances in the method of mixtures.

To elaborate on the method of mixtures, it has long been known that a finite sum or a $\sigma$-finite mixture of NSMs is again a NSM, and this was most notably exploited by Robbins to design some parametric confidence sequences \citep{darling1967confidence}. These mixtures oftentimes encode the experimenter's prior belief on the parameter. In the language of testing by betting \citep{shafer2021testing}, using the likelihood ratio $\frac{p(x;\theta_1)}{p(x;\theta_0) }$ to test the null hypothesis that $\theta=\theta_0$ can be expressed as follows. A prudent gambler, believing that the true parameter is actually not $\theta_0$ but $\theta_1$, uses one dollar to purchase a financial instrument that yields a value $\prod_{i=1}^t \frac{p(X_i;\theta_1)}{p(X_i;\theta_0)}$, which is a fair purchase under the null hypothesis. The instrument will grow exponentially in value if the true parameter is (close to) $\theta_1$. Using the mixture $\frac{p(x;\theta_1)}{p(x;\theta_0) } + \frac{p(x;\theta_2)}{p(x;\theta_0) }$ is analogous to the investor using two initial dollars to hold both instruments, so that the portfolio value will exponentially grow if \emph{either} of $\theta_1, \theta_2$ is (close to) the true one. This is not a problem because \eqref{eqn:vil} adjusts the wealth by the initial investment, and one can always judge the gambler's evidence against the null hypothesis by the ratio of final wealth to initial investment. The method of mixtures is simply an extension of this simple idea.

A bewildering extension of the aforementioned scheme was developed  by \citet[pp.1400-1401]{robbins1970statistical}, who shows by a concrete calculation that an \emph{infinite} ($\sigma$-finite) mixture of likelihood ratios also satisfies a maximal inequality similar to \eqref{eqn:vil}. Since every likelihood ratio is a martingale with expected value 1,  any (unnormalized) infinite mixture of likelihood ratios will results in a nonintegrable process, instead of a martingale. The maximal inequality is:
\begin{equation}\label{eqn:ext-ville-robbins}
    \Pr[ \exists n \ge m, \  M_n \ge C ] \le \Pr[ M_m \ge C ] + C^{-1}\Exp[ M_m \id_{\{ M_m < C \}} ],
\end{equation}
where $M_n = \int \prod_{i=1}^n  \frac{p(X_i;\theta_1)}{p(X_i;\theta_0)} \d F(\theta_1)$  is a mixture of likelihood ratios over any $\sigma$-finite measure $F$, and $m$ is any nonnegative integer. Notably, the process $M$ can be nonintegrable and the truncated mean $\Exp[ M_m \id_{\{ M_m < C \}} ]$ in \eqref{eqn:ext-ville-robbins} avoids the complications caused by its nonintegrability. 

\cite{lai1976confidence} also uses the same technique, crediting Robbins, but these ideas have not gained much attention since then. Both Lai and Robbins do not appear to make much of a deal about processes like $M$, and they did not define or analyze such processes in any generality beyond the mixture likelihood ratio above, with specific attention to the parametric Gaussian setting where these bounds were studied when $F$ is a flat ``improper'' prior, meaning that $\d F(\theta)=\d\theta$.

\paragraph{\rev{Contextualized} contributions.}
We begin where  \cite{robbins1970statistical} and  \cite{lai1976confidence} left off. We will show that the above technique and inequality holds in greater generality than they noted; \eqref{eqn:ext-ville-robbins} is an instance of what we later call the \emph{extended Ville's inequality} (Theorem~\ref{thm:ext-ville}). 
To warm up, we shall show, by a much shorter proof, that \eqref{eqn:ext-ville-robbins} actually holds for any NSM, not just a mixture of likelihood ratios. This opens the door to many nonparametric settings where reference martingale measures do not exist; we demonstrate an example later on. Theorem~\ref{thm:ext-ville} is a strengthening of the classical Ville's inequality \eqref{eqn:vil} that requires no further assumption. 

More importantly however, we point out that there is a unified, abstract theoretical framework hidden beneath the concrete calculation by \cite{robbins1970statistical} on likelihood ratios. We define and study the class of \emph{nonintegrable} nonnegative supermartingales (also called \emph{extended} NSMs in this paper), elevated to objects that are worthy of study in and of themselves. It turns out that the extended Ville's inequality \eqref{eqn:ext-ville-robbins} holds for this large class of processes (Theorem~\ref{thm:evi-2}). To the best of our knowledge, extended NSMs have not been defined or studied in a dedicated manner before our paper. Indeed, as mentioned before it almost appears contradictory to say that a martingale is not integrable, since that is often the first property used to define martingales (and we return to the different concept of locality later). \rev{The textbook by \cite{stroock2010probability}, for example, formulates the theory of martingales and submartingales \emph{whose negative parts are integrable} in the opening parts of Chapter 5.2. However, most of the classical results that follow, including Doob's optional stopping theorem and martingale convergence theorem, are stated and proved with the additional assumption of (full) integrability in that textbook. We, on the other hand, shall commence with a slightly stronger condition than \cite{stroock2010probability} to break free of the constraint of $L^1$ integrability, namely \emph{nonnegativity}. This empowers us to recover in the nonintegrable regime these central results of the classical martingale theory. Such integrability-removing role of nonnegativity, we note, 
}
is not without precedence: in Tonelli's extension of Fubini's theorem on the switch of integration order, nonnegativity is assumed in lieu of $L^1$ integrability. Not coincidentally, Tonelli's theorem will be a recurring tool in our proofs.




\paragraph{Paper outline.}
The rest of this paper is organized as follows. In Section~\ref{sec:prelim}, we first hasten through the well-known theory of integrable NSMs, the classical Ville's inequality, the arithmetics of $[0,\infty]$, and the (extended) conditional expectation of nonintegrable nonnegative random variables.
In Section~\ref{sec:ensm-new} we describe how the well-studied theory supermartingales is extended to the nonintegrable but nonnegative case while keeping most of its key properties, leading to our full proof of the extended Ville's inequality in Section~\ref{sec:evi}.
An interesting statistical application of nonparametric subGaussian mean estimation will be discussed in Section~\ref{sec:estimation}. Following a recent trend in nonparametric, composite hypothesis testing, we define the notion of an extended e-process in Section~\ref{sec:e-proc}. 
Proofs that are omitted in the sections above are in Section~\ref{sec:kpf} or sometimes \cref{sec:pf}.
We end with a discussion and some open directions in Section~\ref{sec:disc}. We put our fully rigorous, textbook-style theoretical foundations of conditional expectations and supermartingales without integrability in \cref{sec:tech}.

\section{Preliminaries}\label{sec:prelim}

In most of this paper, we restrict ourselves to discrete time $n \in \mathbb N_0 = \{0,1,2,\dots\}$, usually just written as $n \ge 0$ for clarity and simplicity.
We fix a probability space $(\Omega, \mathcal{A}, \Pr)$ and a filtration $\mathcal F = \{ \mathcal{F}_n \}_{n \ge 0}$, where $\mathcal F_0   \subseteq \mathcal{F}_{1} \cdots \subseteq \mathcal{A}$.
We call $X = \{ X_n \}_{n \ge 0}$ a stochastic process adapted to $\mathcal{F}$ if $X_n$ is $\mathcal F_n$-measurable for every $n$. Processes and filtrations are indexed by $n = 0,1,2,\dots$ by default unless otherwise noted.
We shall make a few forays into the continuous-time case, where we would temporarily switch from $n \ge 0$ to $t \ge 0$ to distinguish between discrete time and continuous time.


\subsection{Review of integrable nonnegative supermartingales}
\label{sec:l1}

Let $M$ be an $L^1$ stochastic process adapted to $\mathcal{F}$, meaning each $\Exp[|M_n|] < \infty$.
Recall that $M$ is called a 
supermartingale if it satisfies $\Exp[M_n | \mathcal{F}_{n-1}] \le M_{n-1}$ for every $n \ge 1$; and it is called a submartingale or martingale if ``$\le$'' is replaced by ``$\ge$'' and ``$=$'' respectively. 

A variety of maximal inequalities have been established for martingales, supermartingales, and submartingales, especially when the process has a uniform lower bound. For example, an inequality by~\cite{ville1939etude} states that, if $M$ is a nonnegative supermartingale (NSM), then
\begin{equation}\label{eqn:vil-maximal}
    \Pr\left[ \sup_{ 0 \le n < \infty } M_n \ge C \right] \le C^{-1} \Exp[M_0].
\end{equation}
The time-uniform nature of Ville's inequality \eqref{eqn:vil-maximal} makes it the fundamental building block of many recent advances in sequential nonparametric statistics. 
However, Ville's inequality \eqref{eqn:vil-maximal} is not the tightest as its expectation term $\Exp[M_0]$ can be improved to be a \emph{truncated} expectation, which is what we shall call the \emph{extended Ville's inequality} for NSMs. 
\begin{theorem}[Extended Ville's inequality; integrable case]\label{thm:ext-ville}  Let $M$ be a nonnegative supermartingale and $C>0$, $m \ge 0$. Then,
\begin{equation}\label{eqn:ext-ville}
    \Pr[ \exists n \ge m,  M_n \ge C ] \le  \Pr[ M_m \ge C ] + C^{-1}\Exp[ M_m \id_{\{ M_m < C \}} ] = C^{-1}\Exp[ M_m \wedge C ]. 
\end{equation}
\end{theorem}
A special case of the above theorem was presented by \cite{robbins1970statistical}  for mixtures of likelihood ratios, and he anticipated (but did not formally present) a generalization to NSMs. Nevertheless, the above theorem is superceded by our generalization in~Theorem~\ref{thm:evi-2}.

{Theorem~\ref{thm:ext-ville} can simply be proved by applying the ordinary Ville's inequality \eqref{eqn:vil} on the NSM $\{ M_n \id_{\{ 
M_m < C \}} \}_{n \ge m}$. Since it is a special case of Theorem~\ref{thm:evi-2}, we omit the proof.}
In what follows, we shall see that the truncated mean in  Theorem~\ref{thm:ext-ville} enables handling nonintegrable ``supermartingales''.

\subsection{Arithmetics of $[0, \infty]$}
\label{sec:arith}
In abstract analysis, it is often the case that the inclusion of infinities $\pm \infty$ enables one to conveniently formulate the theory in greater generality, and it has been a common practice in numerous measure theory textbooks \citep{klenke2013probability, conway2012course, tao2011introduction} to adopt the set of \emph{extended real numbers} $\overline{\mathbb R} = \mathbb R \cup \{ \pm 
\infty \}$ as the working set of numbers. For our extension, we work with the extended \emph{nonnegative} real numbers $\overline{\mathbb R^+} =  [0, \infty]$, and shall allow
random variables and processes to take values in $\overline{\mathbb R^+}$.

We briefly overview the standard arithmetics and measure-theoretic properties of $\overline{\mathbb R^+}$. The order $\le$ is extended completely to $\eR$ by $x \le \infty$ for any $x \in \eR$. Addition is extended as $\infty + x = x + \infty = \infty$ for any $x \in \eR$, but without the cancellation law{, we therefore need to verify finiteness for every subtraction}.
While $x \cdot \infty = \infty \cdot x = \infty$ holds for any positive $x$, we further \emph{stipulate} that $0 \cdot \infty = \infty \cdot 0 = 0$, as is standard in abstract measure theory. For division, we shall specify the quotient each time we encounter a possible division by 0 or $\infty$. The distributivity of multiplication over addition still holds. The addition or multiplication by a same constant still preserves the $\le$ order. Note that the identity $\left(\lim_{n \to \infty} a_n \right) \left( \lim_{n \to \infty} b_n \right) = \lim_{n \to \infty} a_nb_n$ does not hold in general, but holds when both sequences are \emph{increasing} (called the upward continuity of multiplication) --- for example, infinite sums are upward limits, so $y\sum_i x_i = \sum_i y x_i$ always holds. We shall be very careful when manipulating limits that are not upward in the entire paper.
When speaking of the measurability of $\eR$-valued functions, the Borel $\sigma$-algebra on $\eR$ is the one generated by all intervals $[x, y]$, $x, y \in \eR$. The reader is referred to \citet[Preface]{tao2011introduction} for a more detailed justification of the system $\eR$. We also remark that many important theorems in measure theory hold in $\overline{\mathbb R}$, and hence  in $\eR$. The most important ones in this paper include 
\begin{itemize}
    \item 
the measurability of the pointwise supremum of measurable functions (see e.g.\ \citet[Theorem 1.92]{klenke2013probability}),
\item Beppo Levi theorem of monotone convergence (the upward limit of $\overline{\mathbb R^+}$-valued functions commutes with integral, see \citet[Theorem 4.20]{klenke2013probability}), 
\item Fatou's lemma (for functions taking values in $\overline{\mathbb R^+}$, the ``$\liminf$'' of integrals is larger than the integral of ``$\liminf$''s, see \citet[Corollary 1.4.47]{tao2011introduction}), and 
\item Tonelli's theorem (Fubini's switch of integration order theorem on functions taking values in $\overline{\mathbb R^+}$; see e.g.\ \citet[Theorem 14.16]{klenke2013probability}). 
\end{itemize}
All four shall be used throughout our development either in the subsequent sections or in \cref{sec:tech}. To conclude, the inclusion of $\infty$ to nonnegative reals does not bring any substantial difference in the way we manipulate numbers, except for a few caveats mentioned above.

\subsection{Extended conditional expectation}

Let $X$ be a random variable and $\mathcal{F}$ a sub-$\sigma$-algebra of $\mathcal{A}$. The theory of martingales as is most commonly known is based on the conditional expectation $\Exp[X|\mathcal{F}]$ defined \emph{when $X$ is integrable}, i.e.\ $\Exp[|X|]<\infty$, which leads to those results (among many others) in Section~\ref{sec:l1}. 

Our extension begins with the fact that  $\Exp[X|\mathcal{F}]$ can still be defined when the random variable $X$ takes values in $\overline{\mathbb R^+} = [0,\infty]$ and is not necessarily integrable --- this fact is sometimes mentioned in probability textbooks without much elaboration, but given its centrality in the current paper, we elaborate on it briefly. Multiple ways to define $\Exp[X|\mathcal{F}]$ for nonnegative ($\eR$-valued), possibly nonintegrable $X$ exists. One may define it, for example, as the limit
\begin{equation}\label{eqn:cond-exp-text}
    \Exp[X|\mathcal{F}] := \lim_{n\to\infty}  \Exp[X_n |\mathcal{F}],
\end{equation}
where $\{X_n\}$ is a sequence of integrable nonnegative random variables that converges \emph{upwards} to $X$. One can actually show that such limit does not depend on the choice of the sequence $\{ X_n \}$ \citep[Remark 8.16]{klenke2013probability}. Alternatively, one can define $\Exp[X|\mathcal{F}]$ as the nonnegative random variable $Y$ such that $\Exp[ \id_A Y ] = \Exp[\id_A X]$ for any $A \in \mathcal{A}$, where both sides may equal $\infty$; or as a Radon-Nikodym derivative.
We show in \cref{sec:condexp} that all these ways to define $\Exp[X|\mathcal{F}]$ are equivalent, while sticking to \eqref{eqn:cond-exp-text} with the particular choice $X_n = X \wedge n$ as the working definition of $\Exp[X|\mathcal{F}]$ in our paper.
The arithmetics of $\eR$ in Section~\ref{sec:arith} enable us to carefully verify that the conditional expectation $\Exp[X|\mathcal{F}]$ in this \emph{extended} regime enjoys (nontrivially) all the usual properties of its integrable counterpart.
These include:
\begin{itemize}
    \item $\Exp[ X Y |\mathcal{F} ] = Y \Exp[ X | \mathcal{F} ]$ if $Y$ is $\mathcal{F}$-measurable (\cref{prop:properties}).
    \item Monotone convergence theorem (\cref{prop:cond-beppo-levi}).
    \item Jensen's inequality (\cref{prop:jensen}).
    \item $Y \le \Exp[X|\mathcal{F}]$ if and only if $\Exp[ \id_A Y ] \le \Exp[\id_A X]$ for any $A \in \mathcal{A}$ (\cref{prop:cond-exp-order-rule}).
\end{itemize}
The theory of extended conditional expectation with the all the propositions we prove,
lays the foundation of our upcoming study of nonintegrable supermartingales in $\eR$.


\section{Extended nonnegative supermartingales}
\label{sec:ensm-new}

Much of the theory of nonnegative supermartingales still holds if integrability assumption is dropped, which we develop in full rigor in \cref{sec:mtg}.

\begin{definition}[Extended nonnegative supermartingales]\label{def:ensm}
Let $M$ be a stochastic process taking values in $\overline{\mathbb{R}^+}$, adapted to a filtration $\mathcal{F}$. We call it an extended nonnegative supermartingale (ENSM for short) if, for all $n \ge 1$,
\begin{equation}\label{eqn:def-ensm}
    \Exp[M_n | \mathcal{F}_{n-1}] \le M_{n-1}, \quad \Pr\text{-almost surely}.
\end{equation}
\end{definition}

When we speak of an ENSM henceforth, we by default allow it to take values in $\overline{\mathbb R^+}$, unless otherwise stated. We shall refer to integrable nonnegative martingales and supermartingales as ``classical NMs'' and ``classical NSMs'' for clarity. As in the classical case, if equality holds in \eqref{eqn:def-ensm}, the process is referred to as an extended nonnegative martingale (ENM for short). 
ENSMs have numerous common properties and connection with classical NSMs:
\begin{itemize}
    \item An ENSM $M$ is a classical NSM if and only if $M_0$ is integrable (\cref{prop:finite-moment-ensm-is-nsm}).  
    \item A $\overline{\mathbb R^+}$-valued process is an ENSM if and only if it is an upward limit of NSMs. In fact, for any ENSM $M$, we have two simple but important sequences of classical NSMs that converge upwards to $M$: $\{  M \wedge k  \}_{k \ge 0}$ and $\{ M \id_{ \{ M_0 \le k\} }  \}_{k \ge 0}$ (Propositions~\ref{prop:concave-ensm} and Corollary~\ref{cor:ensm-class}).
    \item Recall that a classical NSM must converge almost surely to an integrable limit due to Doob's martingale convergence theorem. We show that an ENSM $M$ must also converge almost surely, but the limit $M_\infty$ can be nonintegrable and takes values in $\overline{\mathbb R^+}$ (\cref{prop:converge}). In fact, there are ENSMs (and even ENMs) $M$ such that $\Exp[M_n] = \infty$ for any finite $n$ but $M_\infty = 0$ almost surely (Example~\ref{ex:cauchy}).
    \item Like classical NSMs, ENSMs satisfy the optional stopping theorems: if $M$ is an ENSM and $\tau$ is any stopping time, $M^\tau$ is also an ENSM (\cref{prop:stopped-ensm}); for stopping times $\nu_1 \le \nu_2$, $\Exp[M_{\nu_2} | \mathcal{F}_{\nu_2}] \le M_{\nu_1}$ (\cref{prop:optional-sampling}).
\end{itemize}
We encourage the reader to refer to \cref{sec:mtg} in full to see how we carefully develop the theory of ENSMs. Besides those mentioned above, there we also prove some further interesting properties of ENSMs, including
\begin{itemize}
    \item Local ENSMs are ENSMs (\cref{prop:local}). However, we emphasize that local NSMs are a smaller class than ENSMs.
    \item Nonnegative strong solutions of nonpositive-drift SDEs are ENSMs (\cref{prop:sde}).
\end{itemize}

In many regards, one might think of ENSMs are nonnegative processes that, like classical NSMs, expectedly shrink, but without an integrable initial distribution. We further clarify this concept by some elementary examples of ENSMs in the following.

\begin{example}\label{ex:cauchy} Let $Y\sim \operatorname{Cauchy}(0,1)$ be $\mathcal{F}_0$-measurable, and let $\{\xi_n\}_{n\ge 1}$ be a sequence of conditionally $\operatorname{Ber}(\theta)$ random variables adapted to $\{\mathcal{F}_n\}_{n \ge 1}$ (i.e.\ $\Pr[ \xi_n = 1 | \mathcal{F}_{n-1} ] = \theta$ and $\Pr[ \xi_n = 0 | \mathcal{F}_{n-1} ] = 1-\theta$). Let $M_0 =|Y|$, and for $n \ge 1$, let $M_n = M_{n-1} ( 1/2 + \xi_n )$. Then, although $\Exp[M_n] = \infty$, the extended conditional expectation $\Exp[M_n | \mathcal{F}_{n-1}] = M_{n-1}( 1/2 + \theta )$ is well-defined (based on the property (iii) of Proposition~\ref{prop:properties}), and hence the process $M$ is an ENSM adapted to $\mathcal{F}$ if $\theta\le 1/2$, and when $\theta=1/2$, it is an ENM. When this happens, $M_\infty = 0$ (due to martingale strong law of large numbers).
\end{example}


Above, the process is still almost surely finite, we drop this below.

\begin{example}\label{ex:inf-ber}
    Let $N_0 \sim \frac{1}{2}\delta_{1} + \frac{1}{2}\delta_{\infty}$ (i.e.\ $\Pr[N_0  = 1] = \Pr[N_0  = \infty] = 1/2$) be $\mathcal{F}_0$-measurable. Assuming the same $\{ \xi_n \}$ sequence as in Example~\ref{ex:cauchy}, with $N_n = N_{n-1} ( 1/2 + \xi_n )$ we again have $\Exp[N_n | \mathcal{F}_{n-1}] = N_{n-1}( 1/2 + \theta )$ and that $N$ is an ENSM adapted to $\mathcal{F}$ if $\theta\le 1/2$ (and an ENM if $\theta=1/2$). This time, the process is not a.s.\ finite, but it still qualifies as an ENSM under our formulation. $N_\infty$ has distribution $\frac{1}{2}\delta_{0} + \frac{1}{2}\delta_{\infty}$.
\end{example}

The examples above can be thought of the process starting with a heavy-tailed nonnegative distribution, so heavy that the expectation is infinite, and evolving later on {via an ``update rule'' $M_n = M_{n-1} ( 1/2 + \xi_n )$ that would have led to classical supermartingales, if there was integrability}. In \cref{ex:inf-ber}, in each instantiation, $N_n$ is either infinite forever or finite forever (fully determined by $N_0$). In the following example, it can switch from being infinite to being finite at some stopping time.

\begin{example}\label{ex:infinity}
Of course, if for all $n$, $R_n$ almost surely equals a constant $c \geq 0$, then it is a classical NM. But if $R_n = \infty$ almost surely for all $n$, it is an ENM. But ENMs that are not NMs can be more complex. For example, the process defined by $R_0 = \infty, R_n=M_n$, where process $M$ was defined in Example~\ref{ex:cauchy} (with $\theta=1/2$) is also an ENM that starts at $\infty$, becomes a.s.\ finite, and then evolves as a classical NM thereafter.
\end{example}

\begin{example}\label{ex:infinity2}
Consider the same $\{ \xi_n \}$ and $\{M_n\}$ sequence as in Example~\ref{ex:cauchy} with $0 \le \theta \le 1/2$. Suppose the process $T$ starts at $T_0=\infty$ and define $T_n=\infty$ if $\sum_{i=1}^n \xi_i = 0$. Otherwise, let $T_n=M_n$. Then $T$ is an ENSM that starts off infinite but will eventually become finite if $\theta>0$ --- actually $T_\infty = 0$.
\end{example}


\section{The extended Ville's inequality}
\label{sec:evi}

Recall from the beginning of the paper that if $M$ is an \emph{integrable} nonnegative supermartingale, \emph{Ville's inequality} states that  for any $C>0$,
\begin{equation}\label{eqn:ville-inq}
    \Pr[ \exists n, M_n \ge C] \le \Exp[M_0]/C.
\end{equation}
In the following, we present a novel version of Ville's inequality such that (1) it is stronger than \eqref{eqn:ville-inq} in the integrable case; (2) it works for possibly non-integrable, possibly infinite ENSMs we defined.

\begin{theorem}[Extended Ville's inequality; possibly  nonintegrable case]\label{thm:evi-2} Let $M$ be an extended nonnegative supermartingale and $
C > 0$, $m \ge 0$. Then,
\begin{equation}
    \Pr[\exists n\ge m, M_n \ge C] \le \Pr[M_m \ge C] + C^{-1}\Exp[\id_{\{ M_m < C \}} M_m] =  C^{-1}\Exp[ M_m \wedge C ].
\end{equation}
\end{theorem}
In Section~\ref{sec:l1} we presented 
Theorem~\ref{thm:ext-ville}, the ``classical'' special case of Theorem~\ref{thm:evi-2} before introducing non-integrability. The proof of Theorem~\ref{thm:evi-2} {below is also valid for the special case} of Theorem~\ref{thm:ext-ville}.
\begin{proof}[Proof of Theorem~\ref{thm:evi-2}] {Since if we define $M_n ' = M_{m+n}$ and $\mathcal{F}_n' = \mathcal{F}_{m+n}$ then $M'$ is an ENSM on $\mathcal{F}'$,} it suffices to prove the $m=0$ case.
Let $B = \{ M_0 < C \}$. First, $M \id_B$ is an ENSM since $\id_B$ is $\mathcal{F}_0$-measurable, due to \cref{prop:mult-F0}. Next, since $\Exp[M_0 \id_B] \le C < \infty$, due to \cref{prop:finite-moment-ensm-is-nsm} we see that
 $M \id_B$ is always an \emph{integrable} nonnegative supermartingale, regardless whether $M$ is integrable or not. Apply the classical Ville's inequality \eqref{eqn:ville-inq} on $M \id_B $,
\begin{equation}\label{eqn:evipf-ville}
    \Pr[\exists n \ge 0, M_n \id_B \ge C  ] \le C^{-1} \Exp[ M_0 \id_B ].
\end{equation}
Note that
\begin{align}
     \Pr[\exists n \ge 0, M_n \id_B \ge C  ]
     = & \Pr[M_0 < C \wedge \exists n \ge 1, M_n \ge C  ]
    \\ = & \Pr[ \exists n \ge 0, M_n \ge C  ] - \Pr[ M_0 \ge C ],
\end{align}
where the last step uses {$ \Pr[A \cap B] = \Pr[A \cup B^\mathsf{c}] - \Pr[B^\mathsf{c}]$}.
The proof is complete.
\end{proof}

Several remarks are in order. 
First, the truncated mean $\Exp[\id_{\{ M_m < C \}} M_m]$ in Theorem~\ref{thm:evi-2} is always finite, despite the fact that $M_n$ can be nonintegrable. This makes sure that 
Theorem~\ref{thm:evi-2} always gives a non-trivial maximal inequality for an ENSM. 
Let us see in action how our extended Ville's inequality above work on an extended supermartingale which takes $\infty$ with positive probability. 

\begin{example}[Example~\ref{ex:inf-ber} continued]
In Example~\ref{ex:inf-ber}, let $\theta = 1/2$; applying Theorem~\ref{thm:evi-2} with $C = 2$ yields
\begin{equation}
    \Pr[ \exists n \ge 0, N_n \ge 2 ] \le \Pr[ M_0 \ge 2 ] + (1/2)\Exp[ \id_{\{ N_0 < 2\}}  N_0] = 1/2 + 1/4 = 3/4.
\end{equation}    
\end{example}
The above is a nontrivial and accurate conclusion, since the ENM $N$ remains $\infty$ with probability half (thus exceeding $2$), and if $N$ is finite (which also happens with probability half), then it exceeds $2$ with probability $1/2$ by the standard Ville's inequality.

Second, we can alternatively prove Theorem~\ref{thm:evi-2} by applying the extended Ville's inequality for classical NSMs (Theorem~\ref{thm:ext-ville}) on $M \wedge m$ and taking an upward limit $m \to \infty$. This idea of approximating an ENSM with classical NSMs when using the extended Ville's inequality leads to the following statement which links the concentration inequalities (hence also confidence sequences) produced by ENSMs to those by classical NSMs, and it actually only requires the convergence \emph{in distribution} of the \emph{initial} random variables in the processes.

\begin{proposition}\label{prop:vil-conv} 
    Let $ M^{1} , M^{2}, \dots $  and $ M$ be 
    ENSMs, such that 
    $\{ M_{0}^k \}_{k\ge 1}$ converges in distribution to $ M_0 $. {Define the functions $\px_k(C) = C^{-1}\Exp[M_{0}^k \wedge C] $ and $ \px(C) = C^{-1}\Exp[M_{0} \wedge C]$ where $C > 0$.} 
    Then, $\lim_{k \to \infty} \px_k(C) = \px(C)$. Consequently, if all $\px_k$ and $\px$ are bijective and bi-continuous, $\lim_{k\to \infty} 
 \px^{-1}_k(\alpha) = \px^{-1}(\alpha)$ for any $\alpha$ within their common range.
\end{proposition}
The $\px$ functions in the proposition above are so named as to stand for ``{crossing probability bound'' as it follows from Theorem~\ref{thm:evi-2}  that,
    \begin{equation}
    \Pr[\exists n\ge 0, M_{n}^k \ge C] \le  \px_k(C) \quad \text{and} \quad  \Pr[\exists n\ge 0, M_{n} \ge C] \le  \px(C).
    \end{equation}} It is proved in Section~\ref{sec:kpf}.

{Third}, we note that there is another way to avoid the infinite expectation in Ville's inequality for ENSMs.

\begin{theorem}[Ville's inequality on ENSMs for stopping times]\label{thm:evi-stop}
    Let $ M $ be an ENSM. For any stopping times $\pi$ and $\tau$ and constant $C > 0$, $\Pr[ M_{\tau \vee \pi } \ge C  ] \le C^{-1} \Exp[M_\pi]$.
\end{theorem}

\begin{proof}
Due to \cref{prop:optional-sampling}, $\Exp[ M_{\tau \vee \pi} ] \le \Exp[M_{\pi}]$. Applying Markov's inequality on $M_{\tau\vee\pi}$ concludes the proof.
\end{proof}

To see how Theorem~\ref{thm:evi-stop} generalizes the classical Ville's inequality \eqref{eqn:ville-inq} for classical NSMs, we may put $\pi = 0$ and it is known that the resulted stopped bound $\Pr[ M_\tau \ge C ] \le C^{-1} \Exp[M_0]$ for all stopping times $\tau$ is equivalent to the any-time bound $\Pr[ \sup_n M_n \ge C ] \le C^{-1} \Exp[M_0]$ \citep[Lemma 3]{howard2021time}. The role of the stopping time $\pi$ here is actually a ``starting time'', from which we bound the upcrossing probability. In practice, we can choose it so that $\Exp[M_\pi]$, despite the nonintegrability of $M$, is finite. 

\begin{example}[Example~\ref{ex:cauchy} continued]   
For the $M$ in Example~\ref{ex:cauchy}, consider {the hitting time of $[0,2]$, $\pi = \min\{ n : M_n \le 2 \}$}. Then $\Exp[M_\pi] \le 2$.
\end{example}

{Fourth, let us discuss an interpretation of the extended Ville's inequality in financial terms. Consider an efficient, arbitrage-free market where an asset evolves stochastically as an NM. The classical Ville's inequality \eqref{eqn:vil} is often interpreted as the epigram that \emph{you cannot double your wealth with probability more than 1/2} by holding the asset. To set the stage for interpreting the extended Ville's inequality, let us allow the asset $M$ to be an ENSM, and recall the notion of a \emph{covered European call option}, where one already owns the asset $M$, but has the obligation to sell the asset at some \emph{strike price} $K$, on the \emph{expiration day} $m$, if the asset value $M_m > K$. This essentially caps one's final wealth as $M_m \wedge K$. Our \cref{thm:evi-2} can simply be restated as follows.
\begin{corollary}[Extended Ville's inequality for covered calls] Let $\alpha \in (0,1)$.
    The probability of the asset ever valuing more than $C$ after day $m$ is at most $\alpha$, where $C$ is a strike price at which a covered European call option expiring on day $m$ is priced at or less than $\alpha C$.
\end{corollary}
In particular, $C$ always exists and is finite, and is smaller than $\alpha^{-1}\Exp[M_m]$ which may or may not be finite.

Finally, let us mention the recent discovery of the \emph{randomized} Markov's inequality by \cite{ramdas2023randomized} which losslessly tightens the original Markov's inequality by an independent factor of $U\sim \operatorname{Unif}_{(0,1)}$. Let us state a randomized version of \cref{thm:evi-2}. 
\begin{proposition}[Randomized extended Ville's inequality]\label{prop:evi-ur} Let $M$ be an extended nonnegative supermartingale and $
C > 0$, $\tau$ a stopping time adapted to $\mathcal{F}$, and $U\sim \operatorname{Unif}_{(0,1)}$ independent from $\mathcal{F}$. Then,
\begin{equation}
    \Pr\Bigl[\sup_{n \le \tau} M_n \ge C \text{ or } M_\tau \ge UC\Bigr] \le \Pr[M_0 \ge C] + C^{-1}\Exp[\id_{\{ M_0 < C \}} M_0] =  C^{-1}\Exp[ M_0 \wedge C ].
\end{equation}
The proof is by applying the randomized Ville's inequality (Corollary 4.1.1 by \cite{ramdas2023randomized}) instead in step \eqref{eqn:evipf-ville}, to $M\id_B$. We omit the details.
\end{proposition}
}

\section{A nonparametric example of inference with improper priors}
\label{sec:estimation}

``Improper priors'' is a Bayesian term for mixture distributions that are not probability measures, but are $\sigma$-finite. Despite some philosophical debate around their use, these priors are routinely used today in Bayesian inference since they inherently require no subjective information. Here, we show how to incorporate improper priors into non-Bayesian (``frequentist'' or ``game-theoretic'') analyses using the tools developed earlier. 
In particular, the utility of nonintegrable nonnegative supermartingales is greatly expanded by the fact that $\sigma$-finite mixtures of ENSMs, in our example likelihood ratios, are also ENSMs, which we shall state as \cref{lem:mix}. That is, Bayesian improper priors in our examples here serve as mixing distributions over likelihood ratios, leading to ENSMs that are statistically useful in hypothesis testing and uncertainty quantification.

We remark that improper priors are often used by both Bayesians and non-Bayesians to deal with \emph{nuisance} parameters (in both the null and the alternative), leading to a \emph{ratio of improper mixtures} of null and alternative likelihoods; this situation, however, is quite different from the one to be presented in this section, where an improper prior is used on the \emph{main} parameter of interests, leading to an \emph{improper mixtures of the likelihood ratios}. For details on the former case, see papers by, for example \cite{jeffreys1946invariant}, \citet[Section 5]{lai1976confidence}, \cite{berger1998bayes}, \cite{gonen2005bayesian}, and recently \cite{perez2022statistics},
showing that it lends itself in a straightforward way to a non-Bayesian analysis by being interpreted as the (standard) likelihood ratio in a reduced/coarsened filtration. 

We now discuss only the latter of the two cases mentioned above, focusing on the nonparametric problem
of forming a \emph{confidence sequence} (defined soon) for the mean of a sequence of subGaussian random variables. 
The resulting bounds are nearly analytically identical to the (parametric) Gaussian case studied by \cite{robbins1970statistical} and \cite{lai1976confidence}, but since these authors omitted many key steps in the proof, the current section simultaneously serves partly as a pedagogical tool for the Gaussian special case, as well as a new generalization to the subGaussian case which requires some slightly different calculations. 
In \cref{sec:cs-appendix}, we give 3 other ways of deriving such subGaussian confidence sequences (some old, some new) and we contrast the resulting confidence sequences to the one derived below.  

\paragraph{Additional notation.} The following notational conventions will be observed throughout this (and next) section. We use upper case italic letters (e.g.\ $P$) to denote distributions on $\mathbb R$, and calligraphic letters (e.g.\ $\mathcal{P}$) to denote sets of distributions on $\mathbb R$. In particular, we denote by $\normal{\mu}{\sigma^2}$ the normal distribution of mean $\mu$ and variance $\sigma^2$ 
and by $p_{\mu ,\sigma^2}$ its density. {The joint distribution of an infinite i.i.d.\ stream $X_1, X_2,\dots$  of scalar with common distribution $X_1 \sim P$ is written by convention as $P^{\otimes \mathbb N}$.} To allow the stream of data to be non-i.i.d.,
we use blackboard bold italic letters like $\mathbbmsl P$ to denote distributions on $\mathbb R \times \mathbb R \times \dots$, i.e.\ the {joint} distributions of the entire stochastic processes; and Fraktur letters like $\mathfrak P$ for sets of these distributions.
The indexed letters $X_1, X_2, \dots$ are reserved for the data stream, independently and identically distributed according to a distribution that will be clear from context. We denote by $S_n$ the sum $\sum_{i=1}^n X_i$, by $V_n$ the sum of squares $\sum_{i=1}^n X_i^2$. Finally, we define $\avgX{n}= S_n/n$ and $\avgXsq{n} = V_n/n$.

In what follows, we shall frequently say a process is a (super)martingale with respect to a \emph{class of distributions}, {a concept that shall come very useful in statistical applications as the class of distributions is often taken the set of all null distribution in interest. We define this as follows.}  


\begin{definition}[Composite martingales]\label{def:P-mart}
We say that a stochastic process $M$, each $M_n$ being a function of $X_1, \dots, X_n$, is a $\mathbbmsl P$-martingale if, under $\{X_n\}_{n \ge 1} \sim \mathbbmsl P$, it is a martingale.
    If $M$ is a $\mathbbmsl P$-martingale for all $\mathbbmsl P \in \mathfrak{P}$, we call it a $\mathfrak{P}$-martingale.
    We say $M$ is a $P$-martingale if it is a $P^{\otimes \mathbb N}$-martingale; 
    a $\mathcal{P}$-martingale if it is a $P^{\otimes \mathbb N}$-martingale for every $P \in \mathcal P$.
    An analogous definition applies if martingale is replaced by supermartingale or ENSM.
\end{definition}

\begin{definition}[Confidence sequences]\label{def:CS}
Suppose $\mathcal{P}$ is a class of distributions that share a common parameter $\theta_0$. A $(1-\alpha)$-\emph{confidence sequence} (CS) for $\theta_0$ over $\mathcal{P}$ is
a sequence of intervals $\{ \CI_n \}$ such that, for all $P\in \mathcal{P}$,
\begin{equation}
    \Pr[ \forall n, \ \theta_0 \in \CI_n ] \ge 1- \alpha, \quad \text{when }X_1, X_2, \dots \iid P.
\end{equation}  
\end{definition}

The duality between sequential tests for the null $\mathcal{P}$ and CS for $\theta_0$ over $\mathcal{P}$ is a straightforward extension of the duality between tests and CIs in the fixed-time setting. Thus again, Ville's inequality \eqref{eqn:vil} has been the bridge from $\mathcal{P}$-supermartingales to CSs in many cases.

\subsection{The subGaussian class $\subgaussian{\mu_0}{1}$ and confidence sequences for the mean}\label{sec:subG-conf-lr}

Recall that a distribution $P$ on $\mathbb R$ is said to be subGaussian with variance factor 1 (henceforth ``1-subGaussian'') if $\log \Exp[\e^{\lambda(X - \mu_0)}] \le \frac{1}{2}\lambda^2$ for all $\lambda \in \mathbb R$, where $X$ is a random variable distributed according to $P$, and $\mu = \Exp[X]$ is the mean of $P$. We denote by $\subgaussian{\mu}{1}$ the set of all 1-subGaussian distributions with mean $\mu$. Given data $X_1,X_2,\dots$ drawn i.i.d.\ (though the i.i.d.\ assumption can be relaxed) from some distribution $P\in\subgaussian{\mu_0}{1}$ for some unknown $\mu_0$, we would like to sequentially estimate the mean $\mu_0$ by designing a CS for it.
Under the hood, this will require designing a test for whether $P \in \subgaussian{\mu_0}{1}$ against the alternative that $P \in \subgaussian{\mu}{1}$ for some $\mu \neq \mu_0$.


A sequential test for any null $\mathcal{P}$ can be obtained via the classical Ville's inequality \eqref{eqn:vil}; indeed, rejecting the null $\mathcal{P}$ when a nonnegative $\mathcal{P}$-supermartingale first surpasses $1/\alpha$ controls the sequential type I error (the probability of \emph{ever} falsely rejecting the null) to be at most $\alpha$. 
In parametric point-null settings (when $\mathcal P=\{P\}$ is a singleton), the likelihood ratio 
is the most prominent example of such a nonnegative martingale, but sometimes such a process also happens to be a nonnegative supermartingale for a larger class. For example, define
\begin{equation}\label{eqn:gaussian-lr}
    \ell(\mu; \mu_0)_n  :=   \exp\left\{ \sum_{i=1}^n  \frac{(X_i - \mu_0)^2 - (X_i - \mu)^2}{2} \right\}.
\end{equation}
In the Gaussian case, it turns out that $\ell(\mu; \mu_0)_n = \prod_{i=1}^n \frac{p_{\mu,1} (X_i)}{p_{\mu_0,1}(X_i)}$, and is thus a $\normal{\mu_0}{1}$-martingale. More generally, it is known that it is actually a subGaussian supermartingale, which we formally state here and prove in \cref{sec:pf} for the unfamiliar reader.
\begin{lemma}\label{lem:subgaussian-lr}
    For any $\mu, \mu_0 \in \mathbb R$, the process $ \ell(\mu; \mu_0) $ defined in \eqref{eqn:gaussian-lr} is a $\normal{\mu_0}{1}$-martingale and $\subgaussian{\mu_0}{1}$-supermartingale.
\end{lemma}

The task of this section can be formalized as constructing a CS for $\mu_0$ over $\subgaussian{\mu_0}{1}$, and for this Lemma~\ref{lem:subgaussian-lr} leads to many different approaches. However, directly applying Ville's inequality on \eqref{eqn:gaussian-lr} will be undesirable, since as a sequential test for the null $\mu_0$, \eqref{eqn:gaussian-lr} is most powerful only when the actual unknown alternative is (close to) $\mu$.
Thus, in what follows, we will place  a ``\emph{prior}'' on $\mu$, corresponding to mixing (i.e.\ integrating) the process \eqref{eqn:gaussian-lr} over $\mu$ based on some mixing distribution, in order to obtain power for \emph{any} alternative and, correspondingly, a CS that \emph{shrinks}. While classical methods involve mixing the NSM \eqref{eqn:gaussian-lr} with a proper prior (a probability measure) on $\mu$ to obtain a new NSM, our extended mixture lemma presented next ensures that mixing with the NSM \eqref{eqn:gaussian-lr} with an improper prior (a \emph{$\sigma$-finite} measure) will produce an ENSM.

\subsection{Extended method of mixtures: $\sigma$-finite mixtures of ENSMs are ENSMs}
\label{sec:mixture}


If $\{  M(\theta) : \theta \in \Theta \}$ is a family of supermartingales while the parameter set $\Theta$ is equipped with a finite measure $\mu$, it is well known that the \emph{mixture} of these martingales, $\{ \int M(\theta)_n \mu(\d \theta) \}$, is again a supermartingale upon some mild measurability assumptions. We show in this section that not only does the mixture property still hold in our extended, possibly nonintegrable case of (super/sub)martingales, but it also allows for \emph{$\sigma$-finite} mixture thanks to the inclusion of $\infty$ in our arithmetics. 

\begin{lemma}[$\sigma$-finite mixture of extended (super)martingales]\label{lem:mix}
For each $\theta \in \Theta$, let $M(\theta)$ be an ENSM (resp., ENM) 
on a filtered probability space $(\Omega, \mathcal{A}, \mathcal{F}, \Pr)$, indexed by $\theta$ in a measurable space $(\Theta, \mathcal{B})$ such that:
\begin{itemize}
    \item For every $n$, $M(\theta)_n$ is $\mathcal{F}_n\otimes \mathcal{B}$-measurable;
    \item For every $n$, $\Exp[ M(\theta)_n | \mathcal{F}_{n-1} ]$ is $\mathcal{F}_{n-1}\otimes \mathcal{B}$-measurable.
\end{itemize}
Let $\mu$ be a $\sigma$-finite measure on $(\Theta, \mathcal{B})$ and define the mixture $ M^{\mathsf{mix}}_n = \int M(\theta)_n \mu(\d \theta)$. Then, $ M^{\mathsf{mix}}$ is also an ENSM (resp., ENM).
\end{lemma}

The lemma above, as well as its proof, is similar to the classical mixture lemma, and we include it in Section~\ref{sec:kpf} for completeness. In particular, for our composite set-up (Definition~\ref{def:P-mart}) with real parameter $\theta$ {and an index process $M(\theta)$ that depends on data $X_1,X_2,\dots$}, we immediately have the following.

\begin{corollary}
    Suppose $M(\theta)_n = f_n(\theta, X_1, \dots, X_n)$ such that $f_n: \mathbb R^{n+1} \to \eR$ is measurable for each $n$ and each $M(\theta) $ is a $\mathfrak{P}$-ENSM (resp., $\mathfrak{P}$-ENM). Then, $ M^{\mathsf{mix}}$ is also a $\mathfrak{P}$-ENSM (resp., $\mathfrak{P}$-ENM).
\end{corollary}

\subsection{Warm-up: mixing with proper prior $\normal{\mu_0}{1/c^2}$}
\label{sec:mix-nmu0c}

We first present a fairly standard method (dating back to Robbins' work in the late 1960s) to construct CS that mixes the Gaussian likelihood ratio $\ell(\mu; \mu_0)$ \eqref{eqn:gaussian-lr} with $\mu \sim \normal{\mu_0}{1/c^2}$ as the prior. 
\begin{proposition}\label{prop:gaussian-mixed}
    For any $c>0$, the process
    \begin{equation}\label{eqn:guassian-mixed-mtg}
        L_n^{(c)} = \sqrt{\frac{c^2 }{c^2 + n}}\exp\left( \frac{n^2 (\mu_0 - \avgX{n}) ^2 }{2(c^2 + n)} \right)
    \end{equation}
    is a $\normal{\mu_0}{1}$-martingale and $\subgaussian{\mu_0}{1}$-supermartingale and consequently, the intervals 
    \begin{equation}\label{eqn:gaussian-prior-known-var-cs-correct}
    \left[ \avgX{n} \pm \sqrt{ \frac{ \left( \log \frac{c^2 + n}{c^2\alpha^2} \right)(1+c^2/n)}{n} } \right]
\end{equation}
    form a $(1-\alpha)$-CS for $\mu_0$ over $\subgaussian{\mu_0}{1}$.
\end{proposition}
It is not hard to see that the confidence sequence \eqref{eqn:gaussian-prior-known-var-cs-correct}, in the $P = \normal{\mu_0}{1}$ case, is equivalent to the time-uniform concentration bound for Gaussian random walk by \citet[Inequality (17)]{robbins1970statistical}, obtained via a similar mixture method. We remark that the method has a free parameter $c$, arising as the precision of the prior distribution, that can control the radii in \eqref{eqn:gaussian-prior-known-var-cs-correct} for different $n$: \eqref{eqn:gaussian-prior-known-var-cs-correct}'s radius at a fixed $n$ is minimized with $c^2 = \frac{n}{-W_{-1}(-\alpha^2/\e) -1}$. Here $W_{-1}$ is one of the Lambert $W$ functions (see \cref{sec:W}).

To peek into the improper mixture in the next subsection, let us divide the process $L^{(c)}$ by $c$, getting
\begin{equation}\label{eqn:gaussian-mixed-by-c}
    L_n^{[c]} = \sqrt{\frac{1}{c^2 + n}}\exp\left( \frac{n^2 (\mu_0 - \avgX{n}) ^2 }{2(c^2 + n)} \right),
\end{equation}
which corresponds to mixing $\ell(\mu; \mu_0)$ \eqref{eqn:gaussian-lr} with a measure $\mu \sim \frac{1}{c} \normal{\mu_0}{1/c^2}$. As $c$ decreases to 0, the measure $\frac{1}{c} \normal{\mu_0}{1/c^2}$ monotonically increases to the uniform (or ``flat'') measure on $\mathbb R$ with thickness $1/\sqrt{2\pi}$. Taking the limit $c \to 0$, the process upwards approaches a $\subgaussian{\mu_0}{1}$-ENSM, due to Corollary~\ref{cor:ensm-class}. This limiting ENSM is exactly the one arising from a flat mixture which we shall present immediately next.


\subsection{Improper mixing distribution: flat prior $\d \mu$} \label{sec:mix-flt}

{Recall from \cref{sec:mixture} that mixing an NSM with a $\sigma$-finite distribution leads to an ENSM.}
Now, we mix the likelihood ratio $\ell(\mu; \mu_0)$ \eqref{eqn:gaussian-lr} over $\mu$ with a flat improper prior of thickness $1/\sqrt{2\pi}$, that is, $\d F(\mu) / \d \mu = 1/\sqrt{2\pi}$. {The thickness $1/\sqrt{2\pi}$ here is chosen to simplify calculations and, we shall see later, without loss of generality.} We obtain a nonintegrable ENSM, due to \cref{lem:mix}, on which we may utilize our extended Ville's inequality to obtain confidence sequences:

\begin{proposition}\label{prop:flat-mix-cs}
The process
\begin{equation}\label{eqn:flat-mixed-ensm}
    K_0 = \infty, \quad K_n = \frac{1}{\sqrt{n}} \exp\left( \frac{1}{2}n (\mu_0 - \avgX{n})^2 \right)
\end{equation}
is a $\normal{\mu_0}{1}$-ENM and a $\subgaussian{\mu_0}{1}$-ENSM. 
\end{proposition}

Some important remarks are in order. First, as noted before, the ENSM relates to the Gaussian-mixed process in \eqref{eqn:gaussian-mixed-by-c} by $K_n = \lim_{c\to 0}  L_n^{[c]}$. Second,
the expression of $K_n$ is particularly interesting in the Gaussian case. $K_n = \frac{1}{\sqrt{n}} \exp(Z_n^2/2)$ where $Z_n$ has \emph{standard} normal distribution regardless of $n$. If we look at $\exp(Z_n^2/2)$, its distribution is one without mean --- recall the domain of the moment generating function of a $\chi^2$ random variable is $(-\infty, 1/2)$, meaning that $\exp(\lambda Z^2_n)$ has an expectation only for $\lambda < 1/2$. We thus name the distribution of $\exp(Z^2_n/2)$ the \emph{standard critical log-chi-squared distribution}, as $\lambda = 1/2$ is exactly where $\exp(\lambda Z^2_n)$ becomes nonintegrable. Such distribution scaled by a factor of $1/\sqrt{n}$ becomes $K_n$'s distribution, meaning that distribution-wise the process $ K$ is strictly shrinking --- from $K_0 = \infty$ --- a seemingly contradiction with the fact that it is an ENM. Indeed, this phenomenon is never possible for classical martingales, and it is possible for ENMs because they mostly exhibit supermartingale-type behaviors. 

We further remark that, still in the Gaussian case, the process $\{K_n\}_{n \ge 1}$ can be seen as the discretization to a continuous-time process $K_t = g(t, B_t)$ where $\{B_t\}_{t \ge 1}$ is the Brownian motion, and it is then easy to find via It\^o's lemma that $\{K_t\}_{t \ge 1}$ satisfies a drift-free SDE, which implies, thanks to \cref{prop:sde}, $\{K_n\}_{n \ge 1}$ is a $\normal{\mu_0}{1}$-ENSM. It also follows that $\{ K_n - K_1 \}_{n \ge 1}$ is a local martingale.
In \cref{sec:test-sde} we provide the details of the discussion on SDE above, relating it to the previously stated Gaussian mixture Proposition~\ref{prop:gaussian-mixed}.
That it is actually a $\subgaussian{\mu_0}{1}$-ENSM relies on the mixture property (Lemma~\ref{lem:mix}) of ENSMs, and we can see in its proof in Section~\ref{sec:kpf} that $K_n$ equals $(2\pi)^{-1/2} \int_{\mathbb R} \ell(\mu;\mu_0)_n \d \mu$ by a direct calculation.

Finally, we remark on the following game-theoretic property of the process $ K$: it is a manifestation of evidence against the null $\subgaussian{\mu_0}{1}$ as it shrinks to 0 under $\subgaussian{\mu_0}{1}$ and grows to $\infty$ if otherwise. Formally, we have the following statement (proved in Section~\ref{sec:kpf}) which is also satisfied by the classical, Gaussian-mixed NSM $\{ L_n^{(c)} \}$ \eqref{eqn:guassian-mixed-mtg}.
\begin{proposition}\label{prop:Kn-shrink-grow}
    The process $K$ in \eqref{eqn:flat-mixed-ensm} converges almost surely to $K_\infty = 0$ under $\subgaussian{\mu_0}{1}$. In contrast, it converges almost surely to $K_\infty = \infty$ under $\subgaussian{\mu}{1}$ for any $\mu \neq \mu_0$.
\end{proposition}


The ENSM leads to confidence sequences that have slightly different widths for Gaussian and subGaussian data.

\begin{corollary}\label{cor:Gaussian-CS-flat}
     Letting $a_\alpha$ be the solution (which must uniquely exist) in $a$ to the equation $1 - \erf(\sqrt{\log(1/a)}) + 2{a} \sqrt{\frac{ \log (1/a)}{\pi}} = \alpha$, the
following is a $(1-\alpha)$-CS for $\mu_0$ over {$\{ \normal{\mu_0}{1} \}$}:
\begin{equation}\label{eqn:flat-prior-known-var-cs}
    \left[ \avgX{n} \pm \sqrt{\frac{\log( n / a_\alpha^2)}{n}} \right].
\end{equation}
Further, as $\alpha \to 0$, its radius grows asymptotically as
\begin{equation}\label{eqn:flat-prior-known-var-cs-asyp}
    \sqrt{\frac{\log( n / a_\alpha^2)}{n}} \approx  \sqrt{\frac{\log(2n/\pi\alpha^2) }{2n}}.
\end{equation}
\end{corollary}

\begin{corollary}\label{cor:subG-CS-flat}
    Letting $b_\alpha$ be the solution (which must uniquely exist) in $b \in (0,1)$ to the equation $3b + 2b\log(1/b) = \alpha$, the
following is a $(1-\alpha)$-CS for $\mu_0$ over {$\subgaussian{\mu_0}{1}$}:
\begin{equation}\label{eqn:flat-prior-known-var-cs-subg}
    \left[ \avgX{n} \pm \sqrt{\frac{\log( n / b_\alpha^2)}{n}} \right].
\end{equation}
Further, as $\alpha \to 0$, its radius grows asymptotically as
\begin{equation}\label{eqn:flat-prior-known-var-cs-asyp-subg}
    \sqrt{\frac{\log( n / b_\alpha^2)}{n}} \approx  \sqrt{\frac{\log(4n/\alpha^2) + 2\log \log (1/\alpha)  }{n}}.
\end{equation}
\end{corollary}

We note that in the corollaries above, the extended supermartingale $ K$ leads to a slightly smaller confidence sequence \eqref{eqn:flat-prior-known-var-cs} for the parametric Gaussian case, which is equivalent to the time-uniform concentration bound of Gaussian random walk, stated without proof by \citet[Inequality (20)]{robbins1970statistical}; and a slightly wider confidence sequence \eqref{eqn:flat-prior-known-var-cs-subg} for the general, nonparametric subGaussian case. Also, the ``thickness'' of the flat prior does not matter and different choices lead to the same expression of the CSs, which, not obvious \emph{a priori}, shall be demonstrated in the proof (in Section~\ref{sec:kpf}; see also \cref{sec:bayes}). The thickness $1/\sqrt{2\pi}$ here is chosen so as to simplify the expression of $K$.
Thus, the above CSs have no tunable free parameters and although they arise from putting priors on $\mu$, they are flat (often termed \emph{uninformative}) priors. 


The asyptotics \eqref{eqn:flat-prior-known-var-cs-asyp} and \eqref{eqn:flat-prior-known-var-cs-asyp-subg} show that the CSs derived via a flat mixture \eqref{eqn:flat-prior-known-var-cs} still have the polylogarithmic growth rate in $1/\alpha$ when $\alpha \to 0$; like previous CS, the growth rate is also $\sqrt{\log(1/\alpha)}$ when $\alpha$ is sufficiently small. Also, it is clear from our derivation that $\log(1/\alpha)$, $ \log(1/a_{\alpha})$, and $ \log(1/b_{\alpha})$ are approximately of similar scale. 

We shall discuss in \cref{sec:bayes} how
our \cref{cor:Gaussian-CS-flat,cor:subG-CS-flat} are not recoverable without the extended Ville's inequality by
letting the prior precision $c^2 \to 0$ in \cref{prop:gaussian-mixed}, after reviewing some classical Bayesian counterparts of our approaches so far.









\section{Extended e-processes}
\label{sec:e-proc}

{NMs and NSMs are usually derived for sequential tests for their behavior under the null according to Ville's inequality, and often exponential growth under the alternative. However, it is known that for some classes of testing problems, NSMs are unable to distinguish null and alternative distributions, and a larger class of processes, namely \emph{e-processes}, has been proposed and studied \citep{ramdas2022testing,ramdas2020admissible,ruf2022composite}.} E-processes 
share several common properties with supermartingales, for example, the optional stopping theorem and the classical Ville's inequality, and hence playing the similar roles as the accumulation of evidence against the null in sequential tests.  \cite{ramdas2022testing,ruf2022composite} show that e-processes arise somewhat naturally, when dealing with testing composite nulls (where composite supermartingales do not suffice). {A classical e-process can be defined, among various equivalent ways \citep[Lemma 6]{ramdas2020admissible}, as a nonnegative process upper bounded by an NSM starting from 1 \emph{for every distribution}, and we shall extend this definition to the nonintegrable case as follows.} 

    \begin{definition}
    We say that a nonnegative process $
E$, each $E_n$ being a function of $X_1, \dots, X_n$, is a $\mathfrak P$-extended e-process, if, for any $\mathbbmsl P \in \mathfrak{P}$, there exists a $\mathbbmsl P $-ENSM $M^{\mathbbmsl P} $ such that $E_n \le M_n^{\mathbbmsl P}$ almost surely under $\{X_n\}_{n \ge 1} \sim \mathbbmsl P$. 
\end{definition}

In the integrable case, an e-process for $\mathfrak P$ is defined as  a nonnegative process that, for each $\mathbbmsl P \in \mathfrak P$, is upper bounded by some $\mathbbmsl P$-supermartingale with initial value one. In the integrable case, it turns out that an equivalent definition for an e-process is that it is a nonnegative process whose expectation, at any arbitrary stopping time without restriction and under any $\mathbbmsl P \in \mathfrak P$, is at most one.
The reader may refer to \cite{ramdas2020admissible} for the proof of equivalence between these definitions in the classical case and to~\cite{ramdas2022testing} for an application of e-processes in the classical regime, but our exposition on the extended, nonintegrable regime covers and extends the classical formulation.

It is easy to see that an analog (and consequence) of Lemma~\ref{lem:mix} holds for extended e-processes as well, a fact that we record as a corollary (of Lemma~\ref{lem:mix}) for easy reference: 
\begin{corollary}\label{cor:mixture-eproc}
A $\sigma$-finite mixture of $\mathfrak{P}$-extended e-processes is again a $\mathfrak{P}$-extended e-process.
\end{corollary}

{Our examples of extended e-processes involve the following sets of distributions on $\mathbb R \times \mathbb R \times \dots$. First, consider the case where observations are independent, 1-subGaussian, but not identically distributed; their means are all upper bounded by some constant $\mu_0$, i.e.,}
 \begin{equation}
        \subgaussianproc{\le \mu_0}{1} = \{ P_1 \otimes P_2 \otimes \cdots : \text{each } P_n \in \subgaussian{\mu_n}{1}, \mu_n \le \mu_0 \},
\end{equation}
{Next, let us relax the ``means all upper bounded by $\mu_0$'' restriction in the class above, and assume instead that \emph{running average} of the first $n$ means is upper bounded by $\mu_0$ for all $n$, leading to the larger class}
\begin{equation}
        \subgaussianproc{\mathsf{ra}\le \mu_0}{1} = \{ P_1 \otimes P_2 \otimes \cdots : \text{each } P_n \in \subgaussian{\mu_n}{1}, \tfrac{\mu_1 + \dots + \mu_n}{n} \le \mu_0 \},
\end{equation}

{These classes are first studied by} \citet[Section 3.2]{ruf2022composite} {with classical e-processes}.
{A real-world example where the class $\subgaussianproc{\mathsf{ra}\le \mu_0}{1}$ is of interest, for example, is the ``comparing forecasters by the average score differentials'' set-up proposed by \citet[Section 4]{choe2024comparing}.}
We shall show that, for the two one-sided nulls above, there exists a process which is an extended supermartingale for the smaller class $\subgaussianproc{\le \mu_0}{1}$, and an extended e-process for the larger class $\subgaussianproc{\mathsf{ra}\le \mu_0}{1}$. This is achieved by mixing the ``likelihood ratio'' $\ell(\mu ; \mu_0)$ \eqref{eqn:gaussian-lr} with a flat, improper prior on $[\mu_0, \infty)$. That is, with $\d F(\mu)/\d \mu = \id_{ \{ \mu \ge \mu_0 \} }(2\pi)^{-1/2}$, which leads to the following proposition which we prove in \cref{sec:pf}.

\begin{proposition}\label{prop:flat-mix-cs-one-sided} Define $V(x) = \e^{x^2/2} ( 1 + \erf (x/\sqrt{2}) )$.
    The process
\begin{equation}\label{eqn:half-flat-mixed-eproc}
    Q_n = \frac 1{\sqrt{4n}} V(\sqrt{n}(\avgX{n} - \mu_0 )) 
\end{equation}
is a $\subgaussianproc{\le \mu_0}{1}$-ENSM (thus \emph{a fortiori} a $\subgaussian{\mu_0}{1}$-ENSM as well); and a $\subgaussianproc{\mathsf{ra}\le \mu_0}{1}$-extended e-process. It is \emph{not} a $\subgaussianproc{\mathsf{ra}\le \mu_0}{1}$-extended supermartingale.
\end{proposition}

This extended e-process has a limit behavior similar to that of the ENSM $K$, Proposition~\ref{prop:Kn-shrink-grow}, which we state as follows and prove in Section~\ref{sec:kpf}.
\begin{proposition}\label{prop:Qn-shrink-grow}
    The process $Q$ in \eqref{eqn:half-flat-mixed-eproc} converges almost surely to $Q_\infty = 0$ under $\subgaussianproc{\mathsf{ra}\le \mu_0}{1}$. Further, it converges almost surely to $Q_\infty = \infty$ under
    \begin{equation}
        \subgaussianproc{\mathsf{ra}\gtrsim \mu_0}{1} = \{ P_1 \otimes P_2 \otimes \cdots : \text{each } P_n \in \subgaussian{\mu_n}{1}, \liminf_{n\to\infty}\tfrac{\mu_1 + \dots + \mu_n}{n} > \mu_0 \}.
\end{equation}
\end{proposition}

A side product of Proposition~\ref{prop:flat-mix-cs-one-sided}, requiring somewhat a more involved calculation with the extended Ville's inequality (Theorem~\ref{thm:evi-2}), is a one-sided $\subgaussianproc{\mathsf{ra}\le \mu_0}{1}$-confidence sequence for $\mu_0$ which we state as \cref{cor:one-sided-CS}.

\section{Key proofs}
\label{sec:kpf}

\begin{proof}[Proof of Proposition~\ref{prop:vil-conv}]
    Since $\{ M_{0}^k \}_{k\ge 1}$ converges in distribution to $ M_0 $
    and the function $x\mapsto x \wedge C$ is bounded and Lipschitz, by the Portmanteau theorem, $\Exp[M_{0}^k \wedge C]$ converges to $\Exp[M_{0} \wedge C]$, concluding that $\px_k(C)$ converges to $\px(C)$. The convergence of $\px^{-1}_k(\alpha)$ to $\px^{-1}(\alpha)$ under the stated conditions is a standard calculus result.
\end{proof}

\begin{proof}[Proof of Lemma~\ref{lem:mix}] 
Take any $A \in \mathcal{F}_{n-1}$. Since $M(\theta)_n \ge 0$, we can apply Tonelli's theorem to $M(\theta)_n: \Omega \times \Theta \to \overline{\mathbb R^+}$ and the $\sigma$-finite measure $\Pr|_A \otimes \mu$,
\begin{equation}
    \Exp\left[ \id_{A}\int M(\theta)_n \mu(\d \theta) \right] = \int \Exp\left[\id_{A}  M(\theta)_n \right] \mu(\d \theta) =  \int \Exp\left[ \id_A \Exp\left[  M(\theta)_n  \middle\vert \mathcal{F}_{n-1} \right] \right] \mu(\d \theta).
\end{equation}
Next, since $\Exp\left[  M(\theta)_n  \middle\vert \mathcal{F}_{n-1} \right] \ge 0$,
we again apply Tonelli's theorem to $\Exp\left[  M(\theta)_n  \middle\vert \mathcal{F}_{n-1} \right]$ on $\Pr|_A \otimes \mu$,
\begin{equation}
    \int \Exp\left[ \id_A \Exp\left[  M(\theta)_n  \middle\vert \mathcal{F}_{n-1} \right] \right] \mu(\d \theta) = \Exp \left[ \id_A \int \Exp\left[  M(\theta)_n  \middle\vert \mathcal{F}_{n-1} \right] \mu(\d \theta)  \right].
\end{equation}
Therefore, we have
\begin{equation}
  \forall A \in \mathcal{F}_{n-1}, \quad  \Exp\left[ \id_{A}\int M(\theta)_n \mu(\d \theta) \right] = \Exp \left[ \id_A \int \Exp\left[  M(\theta)_n  \middle\vert \mathcal{F}_{n-1} \right] \mu(\d \theta)  \right].
\end{equation}
Further, by Tonelli's theorem, $\int \Exp\left[  M(\theta)_n  \middle\vert \mathcal{F}_{n-1} \right] \mu(\d \theta)$ is $\mathcal{F}_{n-1}$-measurable. Hence,
\begin{equation}
    \Exp\left[ \int M(\theta)_n \mu(\d \theta) \middle\vert \mathcal{F}_{n-1} \right] =  \int \Exp\left[  M(\theta)_n  \middle\vert \mathcal{F}_{n-1} \right] \mu(\d \theta).
\end{equation}
Therefore, we have
\begin{align}
    \Exp[  M^{\mathsf{mix}}_n | \mathcal{F}_{n-1} ] &= \Exp\left[ \int M(\theta)_n \mu(\d \theta) \middle\vert \mathcal{F}_{n-1} \right] \\&=  \int \Exp\left[  M(\theta)_n  \middle\vert \mathcal{F}_{n-1} \right] \mu(\d \theta) \le \int   M(\theta)_{n-1}   \mu(\d \theta) = M^{\mathsf{mix}}_{n-1}.
\end{align}

The fact that $M^{\mathsf{mix}}_n$ is $\mathcal{F}_{n}$-measurable is again guaranteed by Tonelli's theorem.
Hence $M^{\mathsf{mix}} $ is an ENSM.
\end{proof}

\begin{proof}[Proof of Proposition~\ref{prop:flat-mix-cs}]
Let us calculate the mixture of $\ell(\mu; \mu_0)$ over the flat prior with density $1/\sqrt{2 \pi}$.
A straightforward calculation yields: 
\begin{align}
   & \frac{\int_{\mathbb R} \exp(-\frac{1}{2}\sum (X_i - \mu)^2 ) 1/\sqrt{2\pi} \d \mu }{\exp(-\frac{1}{2}\sum (X_i - \mu_0)^2 )}
   =  \frac{ 1/\sqrt{2\pi} \int_{\mathbb R}\exp( -\frac{1}{2}n \mu^2 + S_n \mu  ) \d \mu }{\exp( -\frac{1}{2}n \mu_0^2 + S_n \mu_0  )}
   \\
   = & \frac{  \sqrt{\frac{1}{n}} \exp(S_n^2/2n)  }{\exp( -\frac{1}{2}n \mu_0^2 + S_n \mu_0  )}
    =  \frac{1}{\sqrt{n}} \exp\left( \frac{1}{2}n \mu_0^2 - S_n \mu_0  + S_n^2/2n  \right)
    \\
    = & \frac{1}{\sqrt{n}} \exp\left( \frac{1}{2}n (\mu_0 - \avgX{n})^2 \right). \label{eqn:flat-mixed-mtg}
\end{align}
This is an ENSM due to Lemma~\ref{lem:mix}, and the above expression agrees with $K_n$.
\end{proof}

\begin{proof}[Proof of Proposition~\ref{prop:Kn-shrink-grow}]
    First, suppose $X_1, X_2, \dots \iid P \in \subgaussian{\mu_0}{1}$. By Hoeffding's inequality,
    $\Pr[ |\mu_0 - \avgX{n}| > a  ] \le 2\exp( -na^2/{2} )$ for any $a>0$.
    Therefore, for any $\varepsilon > 0$,
    \begin{equation}
        \Pr[ K_n > \varepsilon ] = \Pr\left[ |\mu_0 - \avgX{n}| > \sqrt{\frac{2\log(\sqrt{n}\varepsilon)}{n}} \right] \le \frac{2}{\sqrt{n} \varepsilon} \to 0.
    \end{equation}
    Thus $K$ converges to 0 in probability. Note that according to \cref{prop:converge}, the ENSM $K$ converges to some $K_\infty$ almost surely. $K_\infty$ must therefore be 0.

    Now, suppose instead $X_1, X_2, \dots \iid P$ with mean $\Pr[X_i] = \mu$ such that $|\mu - \mu_0| = \Delta_\mu > 0$. By SLLN, $\avgX{n} \to \mu$ almost surely. When  $\avgX{n} \to \mu$ happens,
    \begin{equation}
        \liminf_{n \to \infty} K_n \ge \liminf_{n \to \infty} \frac{1}{\sqrt{n}}\exp( n\Delta_\mu^2/2 ) = \infty,
    \end{equation}
    meaning that $K_n \to \infty$ almost surely.
\end{proof}

The proofs of Corollaries~\ref{cor:Gaussian-CS-flat}~and~\ref{cor:subG-CS-flat} rely on the following two tail bounds.

\begin{lemma}\label{lem:trunc-pdf-bound}
    Suppose $Z \sim \normal{0}{1}$ and $M > 0$. Then, $\Exp  [\id_{\{ 0 \le Z \le M \}} \exp(Z^2/2 )]  =  \frac{M}{\sqrt{{2\pi}} }$.
\end{lemma}

\begin{proof}[Proof of Lemma~\ref{lem:trunc-pdf-bound}]
Denoting $p_{0,1}$ as the probability density function of $N_{0,1}$, we have
\begin{equation}
     \Exp  [\id_{\{ 0 \le Z \le M \}} \exp(Z^2/2 )] = \int_{0}^M \exp( z^2/2 )p_{0,1}(z) \d z  = \int_{0}^M \frac{1}{\sqrt{2\pi}} \d z = \frac{M}{\sqrt{2\pi}},
\end{equation}
as claimed.
\end{proof}

\begin{lemma}\label{lem:subgaussian-psi2-bound}
    Let $Z$ be a mean-zero and 1-subGaussian random variable. Then, for any $M > 0$, we have $\Exp[\id_{\{ |Z| \le M \}} \exp(Z^2/2 )] \le 1+M^2$.
\end{lemma}

\begin{proof}[Proof of Lemma~\ref{lem:subgaussian-psi2-bound}]
   \begin{align}
    & \Exp[\id_{\{ |Z| \le M \}} \exp(Z^2/2 )] = \int_{0}^\infty \Pr[\id_{\{|Z| \le M \}} \exp(Z^2/2 ) > t  ] \, \d t
\\
    \le & 1 + \int_{1}^\infty \Pr[\id_{\{|Z| \le M \}} \exp(Z^2/2 ) > t  ] \, \d t
    \\
    = &  1 + \int_{1}^{\exp (M^2/2) } \Pr[ t< \exp(Z^2/2) \le  \exp (M^2/2) ] \, \d t
    \\
    \le &  1 + \int_{1}^{\exp (M^2/2) } \Pr[\exp(Z^2/2) > t] \, \d t
    \\
    = &  1 + \int_{1}^{\exp (M^2/2) } \Pr\left[ |Z| >  \sqrt{2\log t}\right] \, \d t \label{eqn:before-chernoff}
    \\
    \le &  1 + \int_{1}^{\exp (M^2/2) } 2\exp( - (2\log t)/2 ) \, \d t = 1 + \int_{1}^{\exp (M^2/2) } \frac{2}{t} \d t =  1 +M^2, \label{eqn:after-chernoff}
\end{align} 
as claimed.
\end{proof}

\begin{remark}
    We now remark on the difference growth rates on $M$ between Lemma~\ref{lem:trunc-pdf-bound} and Lemma~\ref{lem:subgaussian-psi2-bound}. The quadratic dependence on $M$ comes from the use of Chernoff bound for any 1-subGaussian random variable from \eqref{eqn:before-chernoff} to \eqref{eqn:after-chernoff}. 
    To see that, consider bounding \eqref{eqn:before-chernoff} when $Z\sim \normal{0}{1}$ with  the Gaussian tail inequality $\Pr[ Z \geq x ] \le \frac{\exp(-x^2/2)}{x\sqrt{2\pi}}$. Then, \eqref{eqn:after-chernoff} would be replaced by
    \begin{equation}
 \text{\eqref{eqn:before-chernoff}} \le   1 + 2\int_{1}^{\exp(M^2/2)} 
\frac{\d t}{t\sqrt{2\pi \cdot 2 \log t}}  =  1 + \sqrt{ \frac{\log (M^2/2)}{\pi} }  = 1 + \frac{M}{\sqrt{\pi}},
    \end{equation}
    which is (up to  a small constant) what we obtained in the Gaussian case.
\end{remark}

\begin{proof}[Proof of Corollary~\ref{cor:Gaussian-CS-flat}]
    
Let us first consider the Gaussian case. Let us consider mixture with density $D/\sqrt{2\pi}$ which yields an ENSM (compare \eqref{eqn:flat-mixed-mtg})
\begin{equation}
  K_n^D =  \frac{D}{\sqrt{n}} \exp\left( \frac{1}{2}n (\mu_0 - \avgX{n})^2 \right) \label{eqn:flat-mixed-mtg-d}
\end{equation}
to see how the CS does not depend on $D$.

Note that for \emph{all} $n \ge 1$, 
\[K_n^D =_{\mathsf{d}}  \frac{D \exp(Z^2/2 )}{\sqrt{n}}, 
\] where $Z \sim \normal{0}{1}$. 
The following straightforward calculation for its tail probability holds: for $x > 0$,
\begin{equation}\label{eq:erftail}
    \Pr[  K_1^D \ge x ] = \Pr\left[  |Z| \ge \sqrt{2 \log (x/D)}  \right] = 1 - \erf(\sqrt{\log(x/D)}).
\end{equation}
Further, due to Lemma~\ref{lem:trunc-pdf-bound},
we have
\begin{align}
    \Exp[ \id_{\{K_1^D < x\}} K_1^D] &= \Exp[ \id_{\{ |Z| < \sqrt{2 \log (x/D)} \}} D \exp(Z^2/2 )]   
 = 2D \sqrt{\frac{ \log (x/D)}{\pi}}. \label{eq:trunc-mean}
\end{align}
We now apply the extended Ville's inequality (Theorem~\ref{thm:evi-2}) 
to the extended NSM $\{K^D_n\}_{n \geq 1}$. Using~\eqref{eq:erftail} and~\eqref{eq:trunc-mean}, we get that 
\begin{equation}
    \Pr[\exists n \ge 1,\  K^D_n > x] \le  1 - \erf(\sqrt{\log(x/D)}) + x^{-1}2D \sqrt{\frac{ \log (x/D)}{\pi}}.
\end{equation}
Now let us show that, the error probability on the right hand side can be any $\alpha \in (0,1)$ by appropriately choosing $x$, i.e.,
\begin{equation}\label{eq:solve-for-x}
  \forall D > 0, \alpha \in (0,1), \ \exists x > 0 \ \text{such that }  \alpha = 1 - \erf(\sqrt{\log(x/D)}) + x^{-1}2d \sqrt{\frac{ \log (x/D)}{\pi}}.
\end{equation}
Indeed, consider the continuous increasing function $f(a) = 1 - \erf(\sqrt{\log(1/a)}) + 2{a} \sqrt{\frac{ \log (1/a)}{\pi}}$ for $0 < a \le 1$. Then $f(1)=1$ and $f(0) \to 0$, so one can always find an $a_\alpha$ for which $f(a_\alpha)$ equals $\alpha \in (0,1)$.
 We then set $x/d = 1/ a_\alpha$ to get
\begin{equation}
\Pr[\exists n \geq 1, \ K^D_n > d / a_\alpha] \leq \alpha.
\end{equation}
The above inequality is remarkably Ville-like.
Substituting the expression for $K^D_n$ above, we have
\begin{equation}
    \Pr \left[\exists n \geq 1,\ \frac{D}{\sqrt{n}} \exp\left( \frac{1}{2}n (\mu_0 - \avgX{n})^2 \right) > D / a_\alpha \right] \leq \alpha,
\end{equation}
yielding the expression \eqref{eqn:flat-prior-known-var-cs}.

To derive the asymptotic expression \eqref{eqn:flat-prior-known-var-cs-asyp}, we use the expansion at $\infty$ of the error function:
\begin{equation}
    1-\erf(x) = \e^{-x^2}\left( \frac{1}{\sqrt{\pi}x} + \mathcal{O}(x^{-3}) \right), \quad x \to \infty
\end{equation}
Therefore, when $\alpha, a_{\alpha} \to 0$,
\begin{equation}
   \alpha =  f(a_\alpha) \approx  \e^{-\log(1/a_\alpha)} \frac{1}{\sqrt{\pi \log(1/a_\alpha) }} 
 + 2a_\alpha\sqrt{\frac{\log(1/a_\alpha)}{\pi}} \approx  2a_\alpha\sqrt{\frac{\log(1/a_\alpha)}{\pi}}.
\end{equation}
Now let ${\log(1/a_\alpha)} = g_\alpha$.
\begin{equation}
    \alpha^2 \approx \frac{4}{\pi} \e^{-2g_\alpha} g_\alpha.
\end{equation}
Referring to the definition of the Lambert $W$ functions in \cref{sec:W}, we see that $g_\alpha \approx -\frac{W_{-1}(-\pi\alpha^2/2)}{2}  \approx \frac{\log(2/\pi \alpha^2)}{2}$, due to \cref{lem:W-approx}.
Consequently, the radius in \eqref{eqn:flat-prior-known-var-cs} scales as
\begin{equation}
    \sqrt{\frac{\log n  + 2 g_\alpha}{n}}  \approx \sqrt{\frac{\log(2n/\pi\alpha^2) }{n}},
\end{equation}
as claimed.
\end{proof}

\begin{proof}[Proof of Corollary~\ref{cor:subG-CS-flat}]
    Now, for the subGaussian case, let us again bound $ \Pr[  K_1^D > x ]$ and \newline $ \Exp[ \id_{\{K_1^D \le x\}} K_1^D] $. Letting $Z$ be the centered 1-subGaussian random variable $X_1 - \mu_0$, we have
\begin{equation}\label{eqn:subgaussian-ClogCsq-tail}
    \Pr[  K_1^D \ge x ] = \Pr\left[  |Z| \ge \sqrt{2 \log (x/D)}  \right] \le 2\exp\left( - \sqrt{2 \log (x/D)}^2 /2 \right) = \frac{2D}{x}
\end{equation}
via a standard subGaussian Chernoff tail bound, and
\begin{equation}
    \Exp[ \id_{\{K_1^D < x\}} K_1^D] = \Exp[ \id_{\{ |Z| < \sqrt{2 \log (x/D)} \}} D \exp(Z^2/2 )] \le  D (1 + 2\log(x/D))
\end{equation}
via Lemma~\ref{lem:subgaussian-psi2-bound}. Therefore, via extended Ville's inequality (Theorem~\ref{thm:evi-2}),
\begin{equation}
    \Pr[\exists n \ge 1, \ K^D_n > x] \le \frac{2D}{x} + x^{-1}D (1 + 2\log(x/D)).
\end{equation}
Similarly as in the proof of Corollary~\ref{cor:Gaussian-CS-flat}, we claim that any $\alpha$ can be the right hand side by choosing the appropriate $x$. Consider this time the continuous increasing function $g(b) = 2b + b(1+2\log(1/b))$, satisfying $g(1) = 3$ and $g(0) \to 0$. Let $b_\alpha$ be such that $g(b_\alpha) = \alpha \in (0,1)$, and set $x/D = 1/b_{\alpha}$, giving rise to
\begin{equation}
     \Pr[\exists n \ge 1: K^D_n > D/b_{\alpha}] \le \alpha.
\end{equation}
Substituting the expression for $K^D_n$ above, we have
\begin{equation}
    \Pr \left[\exists n \geq 1, \frac{D}{\sqrt{n}} \exp\left( \frac{1}{2}n (\mu_0 - \avgX{n})^2 \right) > D / b_\alpha \right] \leq \alpha,
\end{equation}
yielding the expression \eqref{eqn:flat-prior-known-var-cs-subg}. The relation between $b_\alpha$ and $\alpha$ is
\begin{equation}
  \left( -\frac{\alpha}{2b_{\alpha}} \right)  \exp\left( -\frac{\alpha}{2b_{\alpha}} \right) = -\frac{\e^{-3/2} \alpha}{2},
\end{equation}
meaning,
\begin{equation}
    1/b_{\alpha} =  - \frac{2 W_{-1}\left(-\frac{\e^{-3/2} \alpha}{2}\right)}{\alpha} \approx \frac{2\log(1/\alpha)}{\alpha}.
\end{equation}
Consequently, the radius in \eqref{eqn:flat-prior-known-var-cs-subg} scales as
\begin{equation}
    \sqrt{\frac{\log (n/b_\alpha^2) }{n}}  \approx \sqrt{\frac{\log(4n \log^2(1/\alpha)/\alpha^2) }{n}}.
\end{equation}
This completes the proof.
\end{proof}

\begin{proof}[Proof of Proposition~\ref{prop:Qn-shrink-grow}]
    Denote the running average $\frac{\mu_1 + \dots + \mu_n}{n} = \overline{\mu}_n$. 
    
    First suppose the underlying distribution $\mathbbmsl P = P_1\otimes P_2 \otimes \dots \in \subgaussianproc{\mathsf{ra}\le \mu_0}{1}$ with $P_i \in \subgaussian{\mu_i}{1}$, so $\overline{\mu}_n \le \mu_0$ for all $n$. Let us look at the ENSM that upper bounds $\{Q_n\}$. 
    
    Define $\tilde{\ell}(\mu;\mu_0;\mathbbmsl P)_n = \exp \left\{ \sum_{i=1}^n (\mu -\mu_0)(X_i - \mu_i) - \frac{n}{2}(\mu-\mu_0)^2 \right\}$. It is easy to see that it is a $\mathbbmsl P$-supermartingale (see \cref{lem:lr-mod-onesided}). Let 
    \begin{align}
        R_n^{\mathbbmsl{P}} = & \int_{\mu_0}^\infty \tilde{\ell}(\mu;\mu_0;\mathbbmsl P)_n (2\pi)^{-1/2} \d \mu    
        \\
        = & \frac{ (2\pi)^{-1/2} \int_{\mu_0}^\infty\exp( -\frac{1}{2}n \mu^2 + (S_n  - n \overline{\mu}_n + n \mu_0 ) \mu  ) \d \mu }{\exp( \frac{1}{2}n \mu_0^2 + S_n \mu_0 - n\overline{\mu}_n \mu_0 )}
   \\
   = & \frac{  \sqrt{\frac{1}{4n}} \exp( (S_n  - n \overline{\mu}_n + n \mu_0 )^2/2n) \left\{1- \erf\left( \sqrt{n/2}(\mu_0 -(\avgX{n} - \overline{\mu}_n + \mu_0)) \right) \right\}  }{\exp( \frac{1}{2}n \mu_0^2 + S_n \mu_0 - n\overline{\mu}_n \mu_0 )}
    \\
    = & \frac{1}{\sqrt{4n}} \exp\left( \frac{1}{2}n 
    ( \avgX{n}^2 + \overline{\mu}_n^2 -2\avgX{n} \overline{\mu}_n )
    \right) \left\{1- \erf\left( \sqrt{n/2}(\overline{\mu}_n - \avgX{n}) \right) \right\} 
    \\ = &  \frac 1{\sqrt{4n}} V(\sqrt{n}(\avgX{n} - \overline{\mu}_n )) \ge Q_n.
    \end{align}
    (The reader may take a pause to appreciate how $R_n^{\mathbbmsl{P}}$ does not depend on $\mu_0$.)

    Since each of $X_i - \Exp(X_i)$ is 1-subGaussian and they are independent, their average $\avgX{n}-\overline{\mu}_n$ is $1/n$-subGaussian and satisfies the tail bound $\Pr[ \avgX{n}-\overline{\mu}_n > a  ] \le \exp( -na^2/{2} )$ for any $a>0$ \citep[pp.\ 22-25]{blmconcineq}.
    Therefore, for any $\varepsilon > 0$,
   \begin{align}
       & \Pr[R_n^{\mathbbmsl{P}} > \varepsilon ] = \Pr[ \avgX{n} - \overline{\mu}_n > n^{-1/2} \cdot V^{-1}(\sqrt{4n}\varepsilon)  ]
      \\
      < & \exp\left\{ - \frac{ (V^{-1}(\sqrt{4n}\varepsilon))^2 }{2} \right\} \to 0,
   \end{align}
   as $n \to \infty$. This is because $\lim_{x \to \infty} V^{-1}(x) = \infty$. Therefore $\{R_n^{\mathbbmsl{P}} \}$ converges in probability to 0. Since $\{ R_n^{\mathbbmsl{P}} \}$ is a $\mathbbmsl P$-ENSM so it converges to 0 almost surely. Therefore, the nonnegative process $\{Q_n\}$ which is upper bounded by $\{ R_n^{\mathbbmsl{P}} \}$ also converges to 0 almost surely.

    Now, suppose $\liminf \overline{\mu}_n =  \mu_0 + \Delta_\mu > \mu_0$ (i.e.\ the $\subgaussianproc{\mathsf{ra}\gtrsim \mu_0}{1}$ case). By Kolmogorov's strong law of large numbers (since $X_i$'s are independent with variances $\le 1$, see e.g.\ \citet[Theorem 6.7]{jiang2010large}), $ \overline{\mu}_n - \avgX{n} \to 0$ almost surely. When this happens,
    \begin{equation}
        \liminf_{n\to\infty} Q_n \ge \liminf_{n\to\infty} \frac{1}{\sqrt{4n}} V(\sqrt{n}\Delta_\mu) = \infty.
    \end{equation}
    So $Q_n \to \infty$ almost surely.
\end{proof}

\section{Discussion and summary}
\label{sec:disc}

We presented an ``extended'' theory of nonintegrable nonnegative supermartingales (ENSMs). This was accomplished by embracing the arithmetics of $[0,\infty]$, generalizing the classical conditional expectation to the nonintegrable case, and then defining the extended supermartingale property.  We then showed how ENSMs arise when mixing supermartingales with ``improper'' sigma-finite priors, and how our extended Ville's inequality can yield interesting statements about such processes, even in nonparametric settings. We also provided connections to the concept of locality, in particular showing that local ENSMs are ENSMs, and mentioned a connection to solutions of stochastic differential equations with nonpositive drift, which can be found in \cref{sec:test-sde}.

As the application of our theoretical framework, we focus on the sequential mean estimation and testing problem with subGaussian data. We discuss some further connections between ENSMs and classical NSMs and apply them to the same statistical problem in \cref{sec:ensm-to-nsm,sec:cs-appendix}.
The resulting confidence sequences are summarized altogether in Table~\ref{tab:summary-gaussian-cs}.

\begin{table}[!h] \footnotesize
    \centering
    \begin{tabular}{|c|c|c|c|}
    \hline
       \emph{Method}  & \makecell{\emph{ Related }\\\emph{prior result} } & \emph{Tightness} & 
       \makecell{\emph{Free}  \\  \emph{parameters}}  \\
       \hline \hline
     \makecell{$\normal{\mu_0}{1/c^2}$ mixture \\ Ville's inequality  \\ (Proposition~\ref{prop:gaussian-mixed}) \\ \& \\ Flat mixture conditioned \\ Ville's inequality \\ (Supp.\ Section 3.2)  } & \cite[(17)]{robbins1970statistical} & $\mathcal{O}(\sqrt{\log(1/\alpha)}, \sqrt{\log n/n})$ & 1 \\ 
     \hline
     \makecell {Flat mixture \\ Extended Ville's inequality  \\ (Proposition~\ref{prop:flat-mix-cs})} & \makecell{\cite[(20)]{robbins1970statistical} \\ (only Gaussian) }&  $\mathcal{O}(\sqrt{\log(1/\alpha)}, \sqrt{\log n/n})$ & 0 \\ 
     \hline
     \makecell{$\normal{\eta}{1/c^2}$ mixture \\ Ville's inequality \\ (\cref{sec:mix-n0c}) } & N/A & $\mathcal{O}(\sqrt{\log(1/\alpha)}, \sqrt{\log n/n}, |\mu_0 - \eta|)$ & 2 \\ 
     \hline
     \makecell{Flat mixture divided \\ Ville's inequality \\  (\cref{sec:div})} & N/A & $\mathcal{O}(\sqrt{\log(1/\alpha)}, \sqrt{\log n/n})$  & 1 \\
     \hline
     \makecell {Half flat mixture \\ Extended Ville's inequality  \\  (\cref{sec:one-sided-cs})} & N/A &  \makecell{$\mathcal{O}(\sqrt{\log(1/\alpha)}, \sqrt{\log n/n})$ \\ (one-sided)} & 0 \\ 
     \hline
    \end{tabular}
    \caption{Comparison of likelihood ratio mixture-based methods to derive confidence sequences for the 1-subGaussian mean. The first method can be derived equivalently by two ostensibly different approaches.}
    \label{tab:summary-gaussian-cs}
\end{table}

Many open questions remain. For example, Doob's submartingale inequality, which can be thought as the dual to Ville's supermartingale inequality \eqref{eqn:vil}, states that a nonnegative submartingale $M$ satisfies $\Pr\left[ \sup_{n \le N} M_n \ge C \right] \le C^{-1} \Exp[M_N]$. Just like the case of Ville's inequality, the expectation term $\Exp[M_N]$ here can also be refined by a truncated mean, in this case $\Exp[M_N \id_{\{\sup_{n \le N} M_n \ge C\}}]$ \citep[Theorem 2.2.6]{van2013introduction}. Does this also open up a maximal inequality in the extended regime? Apparently, the difficulty is that the truncated expected value may still be $\infty$. The problem of an ``extended Doob's inequality'' is important in that Doob's inequality for submartingales is equivalent to the \emph{reverse} Ville's inequality for \emph{backward} submartingales \citep[First Proof of Theorem 2]{manole2023martingale}, which is already used in statistical applications including divergence estimation \citep{manole2023martingale} and PAC-Bayesian bounds \citep{chugg2023unified} on backward submartingales arising from mixtures. Indeed, one would like to allow these mixtures to be improper as well. Another potential future work lies in the problem of continuous-time ENSMs, which we consciously avoid in \cref{prop:sde}.
 Are there any additional complications introduced in continuous time or do these concepts straightforwardly generalize? We anticipate that some of these questions will be answered in the coming years.


\subsection*{Acknowledgments}
We thank Martin Larsson for helpful conversations.

\bibliography{ninsm}

\newpage

\appendix

\section{Further technical background}\label{sec:tech}

\subsection{Lambert $W$ functions}\label{sec:W}
Consider the equation
\begin{equation}
    y = x\e^x.
\end{equation}
When $x$ ranges from $-\infty$ to $-1$, $y$ decreases from $0$ to $-\e$; when $x$ ranges from $-1$ to $\infty$, $y$ increases from $-\e$ to $\infty$. Therefore, one may define $x$ as a function of $y$ on two branches: when $x \le -1$, $W_{-1}(y) = x$; and when $x \ge -1$, $W_0(y) = x$. Therefore, $W_{-1}$ is defined on $y \in [-\e^{-1}, 0)$ while $W_0$ is defined on $y \in [-\e^{-1}, \infty)$.

The following asymptotic bound holds for $-W_{-1}(-\delta)$ when $\delta \to 0$.
\begin{lemma}\label{lem:W-approx} When $0 < \delta < 0.1$,
    \begin{equation} 
      \frac{1}{2}\log(1/\delta) < -W_{-1}(-\delta) < \log(1/\delta),
    \end{equation}
    and
    \begin{equation}
        \lim_{\delta \to 0} \frac{ -W_{-1}(-\delta)}{\log(1/\delta)} = 1.
    \end{equation}
\end{lemma}



\subsection{Conditional expectation of a nonnegative 
random variable}
\label{sec:condexp}

In this section, we rigorously study the ``extended'' conditional expectation for a random variable taking values in $\eR$. We shall first use the method of upward approximation mentioned in numerous textbooks, for example in \cite[Remark 8.16]{klenke2013probability}, to define it and show the otherwise nontrivial fact that the usual definition of conditional expectation (Proposition~\ref{prop:cond-exp-basic-properties}) is still valid, and is equivalent to ours.

Let $\mathcal{F} \subseteq \mathcal{A}$ be a sub-$\sigma$-algebra. Let $X$ be a random variable taking values in $\overline{\mathbb R^+}$. Regardless of the integrability of $X$ (i.e.\ whether $\Exp[X] = \infty$ or not), we define the extended conditional expectation $\Exp[X|\mathcal{F}]$ as
\begin{equation}\label{eqn:cond-exp}
    \Exp[X|\mathcal{F}] := \lim_{n\to\infty}  \Exp[X \wedge n |\mathcal{F}]. 
\end{equation}
We claim that \eqref{eqn:cond-exp} is well defined and satisfies the usual properties of conditional expectation. \rev{Most of these properties, we note, are also proved by \citet[Chapter 5.1]{stroock2010probability} who defines conditional expectations for a larger class of random variables
\begin{equation}
    \{ X : X \text{ takes values in } (-\infty, \infty] \text{ and } \Exp[X^-] < \infty \},
\end{equation}
but we still prove them in this section for completeness.}

First, we show that it recovers the definition of usual conditional expectation.
\begin{proposition} \label{prop:cond-exp-basic-properties}
$\Exp[X|\mathcal{F}]$ in \eqref{eqn:cond-exp} is well defined, and is unique up to a $\Pr$-negligible set. Further,
\begin{enumerate}
    \item[(i)] $\Exp[X|\mathcal{F}]$ is $\mathcal{F}$-measurable;
    \item[(ii)] For any $\mathcal{F}$-measurable set $A$, $\Exp[X\id_A] = \Exp[ \Exp[X|\mathcal{F}] \id_A ]$.
\end{enumerate}
\end{proposition}
\begin{proof}[Proof of Proposition~\ref{prop:cond-exp-basic-properties}]
First, each $\Exp[X \wedge n |\mathcal{F}]$ is defined as the usual conditional expectation since $\Exp[X \wedge n] \le n < \infty$. Hence each $\Exp[X \wedge n |\mathcal{F}]$ is unique up to a $\Pr$-negligible set $N_n$. Therefore, except on a $\Pr$-negligible set $N = \bigcup_n N_n$, all of $\Exp[X \wedge n |\mathcal{F}]$'s are uniquely defined. Since the sequence of functions $\{ X \wedge n \}$ is pointwise increasing, so is the sequence $\{ \Exp[X \wedge n | \mathcal{F} ] \}$ due to the monotonicity of (usual) conditional expectation. Hence, their pointwise limit $\lim_{n\to\infty}  \Exp[X \wedge n |\mathcal{F}]$ can be uniquely defined in $\overline{\mathbb R^+}$, except on a $\Pr$-negligible set.

For property $(i)$, $\mathcal F$-measurability of $\Exp[X|\mathcal{F}]$ follows immediately from 
the fact that the pointwise supremum of measurable functions is measurable in $\overline{\mathbb R^+}$. 

For property $(ii)$, consider the sequence of functions $\{ (X \wedge n)\id_A \}$. Pointwise, they increase to $X \id_A$. Hence, by Beppo Levi (monotone convergence) theorem,
\begin{equation}
    \Exp[ X \id_A ] = \lim_{n \to \infty} \Exp[(X \wedge n)\id_A ].
\end{equation}
By the property of the usual conditional expectation,
\begin{equation}
   \lim_{n \to \infty} \Exp[(X \wedge n)\id_A ] = \lim_{n \to \infty}\Exp[ \Exp[X \wedge n | \mathcal{F}  ]  \id_A ].
\end{equation}
Now consider the sequence $\{ \Exp[X \wedge n | \mathcal{F}  ] \}$. It increases to $\Exp[X|\mathcal{F}]$. Again by Beppo Levi theorem (this time taking the measure $\Pr|_A$),
\begin{equation}
    \lim_{n \to \infty}\Exp[ \Exp[X \wedge n | \mathcal{F}  ]  \id_A ] = \Exp[ \Exp[X|\mathcal{F}] \id_A ].
\end{equation}
The proof is complete.
\end{proof}

\begin{proposition}\label{prop:uniqueness}
The properties (i) and (ii) in Proposition~\ref{prop:cond-exp-basic-properties} together uniquely, up to a $\Pr$-negligible set, characterize the extended conditional expectation \eqref{eqn:cond-exp}. 
\end{proposition}
\begin{proof}[Proof of Proposition~\ref{prop:uniqueness}] Suppose $Y$ and $Y'$ both satisfy (i) and (ii) in Proposition~\ref{prop:cond-exp-basic-properties}. That is, both are $\mathcal{F}$-measurable, and for any $\mathcal{F}$-measurable set $A$, $\Exp[X\id_A] = \Exp[ Y \id_A ] = \Exp[ Y' \id_A ]$. It suffices to prove that $\Pr[Y > Y'] = 0$. Note that
\begin{equation}
    \{  Y > Y' \} = \bigcup_{n ,m =1}^\infty \{ Y \ge Y' + 2^{-n}  \} \cap \{ Y' \le m \} = \bigcup_{n ,m =1}^\infty A_{n,m}.
\end{equation}
On each $A_{n,m}$ we have,
\begin{equation}
\label{eq:before-cancel}
    \Exp[Y \id_{A_{n,m}}] \ge \Exp[(Y'+2^{-n}) \id_{A_{n,m}}] = \Exp[Y'\id_{A_{n,m}}] + 2^{-n} \Pr[A_{n,m}].
\end{equation}
By our original supposition, we also have $ \Exp[Y \id_{A_{n,m}}] = \Exp[Y'\id_{A_{n,m}}] \le m < \infty$. The finiteness allows us to call the cancellation rule of \emph{real} number addition in~\eqref{eq:before-cancel}. Hence $\Pr[A_{n,m}] \le 0$, which implies $\Pr[A_{n,m}] = 0$. Thus, $\Pr[Y > Y'] = \Pr[\bigcup_{n ,m=1}^\infty  A_{n,m} ] = 0$.
\end{proof}

\begin{remark} Recall that any nonnegative random variable $X$ can be written as a Radon-Nikodym derivative $d\nu/d\Pr$ (indeed, take $\nu(A) = \int_A X d\Pr$). Similarly, recall that the usual conditional expectation of an integrable nonnegative random variable is also an example of a Radon-Nikodym derivative, implying its existence and uniqueness. We remark here that this fact is still true in our extended setting. First, consider the following measure on $(\Omega, \mathcal{F})$:
\begin{equation}
    \nu: A \mapsto \Exp[X \id_A], \text{ for } A \in \mathcal F.
\end{equation}
Despite the fact that $X$ may be nonintegrable,  the measure $\nu$ is always absolutely continuous w.r.t.\ $\Pr$. The seeming issue is, when $X$ is not $\Pr$-a.s.\ finite, $\nu$ can be non-$\sigma$-finite. However, the Radon-Nikodym theorem holds even when the numerator measure is not $\sigma$-finite \citep[Chapter 4.1, Exercise 3(b)]{conway2012course}, as long as we allow the derivative to take value $\infty$. Hence, $\Exp[X|\mathcal{F}]$ is the unique Radon-Nikodym derivative $\frac{\d \nu}{\d \Pr|_{\mathcal F}}$.
\end{remark}

We now proceed to check that the other basic properties of the usual conditional expectation still hold.

\begin{proposition}\label{prop:properties}
For $X, Y$ being any two random variables taking values in $\overline{\mathbb R^+}$, we have
\begin{enumerate}
    \item[(i)] Linearity: Let $u, v\ge0$. Then $\Exp[u X + vY | \mathcal{F}] = u \Exp[X | \mathcal{F}] + v \Exp[Y | \mathcal{F}]$.
    \item[(ii)] Monotonicity: If $X \ge Y$ a.s., then $\Exp[X|\mathcal{F}] \ge \Exp[Y | \mathcal{F}]$.
    \item[(iii)] Measurable product: If $Y$ is $\mathcal{F}$-measurable, then $\Exp[XY | \mathcal{F}] = Y \Exp[X | \mathcal{F}]$.
\end{enumerate}
\end{proposition}

    \begin{proof}[Proof of Proposition~\ref{prop:properties}] Since the first two properties are somewhat simple to prove, we only prove $(iii)$.  Take an arbitrary $A \in \mathcal{F}$. Let $Y_n= \id_{\{ Y < \infty \}} 2^{-n} \lfloor 2^n Y \rfloor +  \id_{\{Y = \infty\}} \infty$, which pointwise increases to $Y$. Hence, the sequence $\{  Y_n \Exp[X|\mathcal{F}]  \}$ pointwise increases to $Y \Exp[X|\mathcal{F}]$, and $\{  Y_n X \}$ pointwise increases to $YX$. By Beppo Levi theorem on the increasing sequences $\{  Y_n \Exp[X|\mathcal{F}]  \}$ and $\{  Y_n X \}$ w.r.t.\ the measure $\Pr|_A$ (defined as $\Pr|_A[B] = {\Pr[A \cap B]}$),
    \begin{gather}
        \lim_{n \to \infty} \Exp[ \id_A Y_n \Exp[X|\mathcal{F}] ] =  \Exp[ \id_A Y \Exp[X|\mathcal{F}] ],
        \\
        \lim_{n \to \infty} \Exp[ \id_A Y_n X ] =  \Exp[ \id_A Y X ].
    \end{gather}
    On the other hand, let $\overline{\mathbb N_0}$ be the extended natural numbers $\{ 0, 1, \dots , \infty \}$. We have
    \begin{align}
     \Exp[ \id_A Y_n \Exp[X|\mathcal{F}] ]
         = &  \Exp\left[ \id_A \left(  \sum_{k \in \overline{\mathbb N_0}} k 2^{-n} \id_{\{Y_n = k2^{-n}\}} \right) \Exp[X|\mathcal{F}] \right] \label{eqn:exp-idA-sum}
        \\
        = & \sum_{k \in \overline{\mathbb N_0}}  \Exp[ \id_A k 2^{-n} \id_{\{Y_n = k2^{-n}\}} \Exp[X|\mathcal{F}] ] \label{eqn:sum-exp-idA}
        \\
        = & \sum_{k \in \overline{\mathbb N_0}}   \Exp[ \id_{A\cap \{Y_n = k2^{-n}\}} k 2^{-n} \Exp[X|\mathcal{F}] ] \label{eqn:sum-exp-idAY-exp}
        \\
        = & \sum_{k \in \overline{\mathbb N_0}}   \Exp[ \id_{A\cap \{Y_n = k2^{-n}\}} k 2^{-n}X ] \label{eqn:sum-exp-idAY}
        \\
        = & \Exp \left[ \sum_{k \in \overline{\mathbb N_0}} \id_{A\cap \{Y_n = k2^{-n}\}} k 2^{-n}X \right]  \label{eqn:exp-sum-idAY}
        \\
        = & \Exp[\id_A Y_n X].
    \end{align}
From \eqref{eqn:exp-idA-sum} to \eqref{eqn:sum-exp-idA}, and from \eqref{eqn:sum-exp-idAY} to \eqref{eqn:exp-sum-idAY}, we use Tonelli's theorem to swap between expectation and summation over nonnegative random variables. From \eqref{eqn:sum-exp-idAY-exp} to \eqref{eqn:sum-exp-idAY}, with $Y_n$ being $\mathcal{F}$-measurable, we use (ii) of Proposition~\ref{prop:cond-exp-basic-properties} on the extended conditional expectation $\Exp[X | \mathcal{F}]$ and $A\cap \{Y_n = k2^{-n}\} \in \mathcal{F}$.

Taking $n \to \infty$ on both sides of $\Exp[ \id_A Y_n \Exp[X|\mathcal{F}] ] =\Exp[\id_A Y_n X]$, we use Beppo Levi theorem once again to have $\Exp[ \id_A Y \Exp[X|\mathcal{F}] ] = \Exp[ \id_A Y X ]$. Since $A \in \mathcal{F}$ is arbitrary, we see that $Y \Exp[X|\mathcal{F}]$ equals $\Exp[YX|\mathcal{F}]$.
    \end{proof}

We have the following extended Beppo Levi's monotone convergence theorem for conditional expectation.
\begin{proposition}[Monotone convergence]\label{prop:cond-beppo-levi}
    Let $X_n$ be a sequence of random variables taking values in $\overline{\mathbb R^+}$ that almost surely increases to $X$. Then, $\lim_{n \to \infty} \Exp[ X_n| \mathcal{F}] = \Exp[ X| \mathcal{F}] $.
\end{proposition}
\begin{proof}[Proof of Proposition~\ref{prop:cond-beppo-levi}] First, the limit must exist since the sequence $\{ \Exp[ X_n| \mathcal{F}] \}$ is increasing.
    For any $A \in \mathcal{F}$,
    \begin{align}
        &\Exp[ \id_A  \lim_{n \to \infty} \Exp[ X_n| \mathcal{F}] ] = \Exp[  \lim_{n \to \infty} \id_A \Exp[ X_n| \mathcal{F}] ] =    \lim_{n \to \infty} \Exp[   \id_A \Exp[ X_n| \mathcal{F}] ]     
        \\= & \lim_{n \to \infty} \Exp[   \id_A X_n] =  \Exp[ \lim_{n \to \infty}  \id_A X_n]  =\Exp[   \id_A\lim_{n \to \infty}  X_n]  = \Exp[\id_A X ].  
    \end{align}
    The first and fifth equalities above are due to the monotonicity of $\{ \Exp[ X_n| \mathcal{F}] \}$ and $\{X_n\}$, the second and the fourth to Beppo Levi theorem for expectation in $\overline{\mathbb R^+}$. This concludes the proof.
\end{proof}

The conditional version of Jensen's inequality works in the extended regime as well.
\begin{proposition}[Jensen's inequality]\label{prop:jensen}
    Let $X$ be a random variable taking values in $\overline{\mathbb R^+}$ and $g: \overline{\mathbb R^+} \to \overline{\mathbb R^+}$ a concave function. Then, $\Exp[ g(X) |  \mathcal{F} ] \le g( \Exp[X | \mathcal{F}] ) $.
\end{proposition}
\begin{proof}[Proof of Proposition~\ref{prop:jensen}]
    Consider $g_n(x) = g(x) \wedge n$, which is concave and $g_n \to g$ upwards. The proposition follows from Jensen's inequality for classical conditional expectation and Proposition~\ref{prop:cond-beppo-levi}.
\end{proof}

Armed with the ability to define conditional expectations for (potentially) nonintegrable random variables, we are now able to formally define the extended (potentially nonintegrable) nonnegative supermartingales mentioned in \cref{sec:ensm-new}. 
The following straightforward fact will be handy in proving some properties of extended nonnegative supermartingales in the next subsection.

\begin{proposition}\label{prop:cond-exp-order-rule}
    Suppose $X, Y$ are random variables taking values in $\overline{\mathbb R^+}$ and $Y$ is $\mathcal{F}$-measurable. Then, $\Exp[X|\mathcal{F}] \le Y$ a.s.\ if and only if for any $A \in \mathcal{F}$, $\Exp[ X \id_{A} ] \le \Exp[ Y \id_{A} ]$.
\end{proposition}

\begin{proof}[Proof of Proposition~\ref{prop:cond-exp-basic-properties}]
    The ``$\implies$'' direction follows directly from (ii) of Proposition~\ref{prop:cond-exp-basic-properties}. Now suppose $\Exp[ X \id_{A} ] \le \Exp[ Y \id_{A} ]$ for any $A \in \mathcal{F}$. Let $Y' = \Exp[X | \mathcal{F}]$. Then $\Exp[ Y' \id_{A} ] \le \Exp[ Y \id_{A} ]$ for any $A \in \mathcal{F}$. The rest is very similar to the proof of Proposition~\ref{prop:uniqueness}. Note that
\begin{equation}
    \{  Y' > Y \} = \bigcup_{n ,m=1}^\infty \{ Y' \ge Y + 2^{-n}  \} \cap \{ Y \le m \} =\bigcup_{n ,m=1}^\infty  A_{n,m}.
\end{equation}
On each $A_{n,m}$ we have,
\begin{align*}
    \Exp[Y' \id_{A_{n,m}}] &\ge \Exp[(Y+2^{-n}) \id_{A_{n,m}}] \\
    &= \Exp[Y\id_{A_{n,m}}] + 2^{-n} \Pr[A_{n,m}] 
    \ge \Exp[Y'\id_{A_{n,m}}] + 2^{-n} \Pr[A_{n,m}].
\end{align*}
Note that $ \Exp[Y'\id_{A_{n,m}}] \le \Exp[Y\id_{A_{n,m}}] \le m < \infty$. The finiteness allows us to call the cancellation rule of {real} number addition. Hence $\Pr[A_{n,m}] \le 0$, which implies $\Pr[A_{n,m}] = 0$. The claim concludes from $\Pr[Y'> Y] = \Pr[\bigcup_{n ,m=1}^\infty  A_{n,m} ] = 0$.
\end{proof}

\subsection{Nonnegative supermartingales without integrability}
\label{sec:mtg}

\subsubsection{Basic properties}

Recall that we have formulated the concept of nonnegative supermartingales in \cref{def:ensm}, dropping the usual assumption of finiteness and integrability. We begin with some properties that connect ENSMs with classical NSMs.

 As it should, an ENSM multiplied by a nonnegative constant is again an ENSM. To generalize a bit, we have the following:
\begin{proposition}\label{prop:mult-F0} Let $M$ be an ENSM adapted to $\mathcal{F}$. Let $Y$ be an $\mathcal{F}_0$-measurable random variable taking value in $\overline{\mathbb R^+}$. Then, the process $ YM$ is an ENSM adapted to the same filtration.
\end{proposition}
\begin{proof}[Proof of Proposition~\ref{prop:mult-F0}] Let $n \ge 1$. Then $Y$ is $\mathcal{F}_{n-1}$-measurable.
\begin{equation}
    \Exp[ YM_{n} | \mathcal{F}_{n-1} ] = Y  \Exp[ M_{n} | \mathcal{F}_{n-1} ]  \le Y M_{n-1}.
\end{equation}
Hence $ YM$ is an ENSM.
\end{proof}

The following proposition determines the ``classicality'' of an ENSM by the integrability of its initial random variable.
\begin{proposition}\label{prop:finite-moment-ensm-is-nsm} Let $M$ be an ENSM adapted to $\mathcal{F}$. It is a classical NSM if and only if $\Exp[M_0] < \infty$.
\end{proposition}
\begin{proof}[Proof of Proposition~\ref{prop:finite-moment-ensm-is-nsm}]
 If $\Exp[M_0] < \infty$, we see that $\Exp[M_n] < \infty$ for any $n$. Since our possibly nonintegrable conditional expectation extends the classical integrable one, the inequality $\Exp[M_n | \mathcal{F}_{n-1}] \le M_{n-1}$ holds in the classical sense as well, concluding that $ M$ is a NSM in the classical sense.
\end{proof}

The following proposition bridges classical supermartingales, not necessarily nonnegative, to ENSMs. Arithmetically, take heed that we shall allow a very minor generalization from our previously introduced system $\overline{\mathbb R^+}$: here we need to add two numbers in $(-\infty, \infty]$, which is of course well-defined (and actually frequently performed in classical probability theory).

\begin{proposition}\label{prop:M0+classical-SM}
    Let $M_0$ be a $\mathcal{F}_0$-measurable nonnegative random variable and let $
Y$ be a classical supermartingale (not necessarily nonnegative) such that $M = M_0 + Y \ge 0 $ almost surely. Then, $ M$ is an ENSM.
\end{proposition}
\begin{proof}[Proof of Proposition~\ref{prop:M0+classical-SM}]
    Let $A \in \mathcal{F}_{n-1}$. Then the inequality
    \begin{equation}\label{eqn:first-step-classical-sm-to-ensm}
        \Exp[ \id_{A} Y_n ] \le \Exp[ \id_{A} Y_{n-1} ]
    \end{equation}
    holds in $\mathbb R$. The expectation $\Exp[ \id_{A} M ]$ is a number in $\overline{\mathbb R^+}$. Since $\id_A Y_n + \id_A M  = \id_A M_n \ge 0$ and $ \id_A Y_{n-1} + \id_A M  = \id_A M_{n-1} \ge 0 $, they have well-defined expectations in $\overline{\mathbb R^+}$, so we can indeed add $\id_A M$ to the expectations of both sides of the inequality \eqref{eqn:first-step-classical-sm-to-ensm}, getting $\Exp[ \id_A M_n  ] \le \Exp[ \id_A M_{n-1} ]$. It follows from Proposition~\ref{prop:cond-exp-order-rule} that $M$ is an ENSM.
\end{proof}

As expected, we may construct a new ENSM by an existing ENSM and a concave, increasing function, just like the case for classical supermartingales.

\begin{proposition}\label{prop:concave-ensm}
    Let $M$ be an ENSM and $g: \overline{\mathbb R^+} \to \overline{\mathbb R^+} $ concave and increasing. Then, $g(M)$ is an ENSM. In particular, for any $x \ge 0$, $M \wedge x$ is a classical NSM.
\end{proposition}
\begin{proof}[Proof of Proposition~\ref{prop:concave-ensm}] Due to Proposition~\ref{prop:jensen},
    $\Exp[g(M_n)|\mathcal{F}_{n-1} ] \le g( \Exp[M_n|\mathcal{F}_{n-1} ] ) \le g(M_{n-1})$. So $g(M)$ is an ENSM.
\end{proof}

The last sentence of Proposition~\ref{prop:concave-ensm} is very powerful, as it will help us extend any property of classical supermartingales that is closed under upward limit to ENSMs in the upcoming subsections. Further, it leads to the following characterization of ENSMs which may serve as their alternative definiton.

\begin{corollary}\label{cor:ensm-class}
    A process is an ENSM adapted to $\mathcal{F}$ if and only if it is an upward limit of classical NSMs adapted to $\mathcal{F}$.
\end{corollary}
\begin{proof}[Proof of Corollary~\ref{cor:ensm-class}]
    If $ M$ is an ENSM, it follows from Proposition~\ref{prop:concave-ensm} that $M \wedge 1 ,  M \wedge 2 ,\dots$ are classical NSMs. Now suppose for each $k$, $M_{(k)} $ is a classical NSM adapted to $\mathcal{F}$, $M_{(k)n} \le M_{(k+1)n}$ almost surely for every $m,n\ge 0$, and $M_n = \lim_{k\to\infty}M_{(k)n}$. First, it is easy to see that $M_n$ is $\mathcal{F}_n$-measurable as it is the pointwise supremum of $\{M_{(k)n}\}_{k\ge0}$. Now we take the upward limit $k\to\infty$ on both sides of $\Exp[ M_{(k)n} | \mathcal{F}_{n-1} ] \le M_{(k)n-1}$. Due to Proposition~\ref{prop:cond-beppo-levi}, we have $\Exp[ M_n | \mathcal{F}_{n-1}] \le M_{n-1}$, concluding the proof.
\end{proof}

\subsubsection{Convergence of extended nonnegative supermartingales}

It is well-known that a classical NSM must convergence almost surely to a random variable \citep[Corollary 11.5]{klenke2013probability}. We show that this holds for ENSMs as well.

\begin{proposition}[Almost sure convergence of ENSMs]\label{prop:converge}
    An ENSM $ M$ converges almost surely to a random variable $M_\infty$ taking values in $\overline{\mathbb R^+}$.
\end{proposition}
\begin{proof}[Proof of Proposition~\ref{prop:converge}]
    For $0 < a < b < \infty$, define the event
    \begin{align}
        D^{a,b} &= \left\{ \liminf_{n \to \infty} M_n < a \right\} \cap \left\{ \limsup_{n \to \infty} M_n > b \right\} \\
        &= \left\{ \liminf_{n \to \infty} M_n \wedge (b+1) < a \right\} \cap \left\{ \limsup_{n \to \infty} M_n\wedge (b+1) > b \right\}.
    \end{align}
    Since the classical NSM $ M \wedge (b+1) $ converges a.s., $\Pr[  D^{a,b} ] = 0$. Therefore, letting
    \begin{equation}
        D = \bigcup_{ \substack{ 0 <a < b< \infty \\ a, b \in \mathbb Q } } D^{a,b},
    \end{equation}
    we have $\Pr[D] = 0$. Noting that $M$ converges on the complement of $D$ concludes the proof.
\end{proof}

\subsubsection{Stopping time properties}

Some of the most important results of classical (super)martingales are related to their behaviors at stopping times. They are referred to as ``optional stopping'' and/or ``optional sampling'' theorems.
With some caution, we recover some of them for ENSMs. Recall that a stopping time $\tau$ is a $\overline{\mathbb N_0}$-valued random variable such that $\{ \tau = n \} \in \mathcal{F}_n$ for every $n$, and the $\sigma$-algebra of $\tau$-past is defined as
\begin{equation}
    \mathcal{F}_\tau = \{ A \in \mathcal{A} : A \cap \{ \tau = n \} \in \mathcal{F}_n \text{ for all }n \}.
\end{equation}

First, we shall show that an ENSM stopped by a stopping time is still an ENSM, an extension of the classical optional stopping theorem.

\begin{proposition}[Optional stopping]\label{prop:stopped-ensm}
    Let $M$ be an ENSM and $\tau$ a stopping time. Then, $M^\tau := \{ M_{n \wedge \tau} \}$ is an ENSM.
\end{proposition}
\begin{proof}[Proof of Proposition~\ref{prop:stopped-ensm}] For any $n$, note that $\{ \tau \ge n \} \in \mathcal{F}_{n-1}$. We have,
    \begin{align}
        & \Exp[M_{n \wedge \tau} | \mathcal{F}_{n-1}  ] =   \Exp\left[M_{n \wedge \tau}\left( \sum_{j=0}^{n-1}  \id_{\{ \tau = j \}} + \id_{\{ \tau \ge n \}}\right) \middle\vert \mathcal{F}_{n-1}  \right]  \\
        = & \sum_{j=0}^{n-1} \Exp[ M_j \id_{\{ \tau = j \}}| \mathcal{F}_{n-1} ] + \Exp[M_n \id_{\{ \tau \ge n \}} |\mathcal{F}_{n-1}]
        \\
        = &  \sum_{j=0}^{n-1} M_j \id_{\{ \tau = j \}} +   \id_{\{ \tau \ge n \}} \Exp[M_n |\mathcal{F}_{n-1}]
        \\
        \le  & \sum_{j=0}^{n-1} M_j \id_{\{ \tau = j \}} +   \id_{\{ \tau \ge n \}}M_{n-1} = M_{(n-1)\wedge \tau}.
    \end{align}
    This concludes the proof.
\end{proof}

We have the following statement which is satisfied by classical supermartingales (see e.g.\ \cite[Theorem 10.11]{klenke2013probability}).
\begin{proposition}[Optional sampling]\label{prop:optional-sampling}
    Let $M$ be an ENSM. Suppose $\nu_1 \le \nu_2$ are stopping times. Then, $\Exp[ M_{\nu_2} | \mathcal{F}_{\nu_1} ] \le M_{\nu_1}$.
\end{proposition}
Before stating our very concise proof, we note that $\nu_2$ does not need to be finite since $M_\infty$ has been established by Proposition~\ref{prop:converge}.

\begin{proof}[Proof of Proposition~\ref{prop:optional-sampling}]
   Consider for $m \in \mathbb N_0$ the classical NSM $ M \wedge m $. We have $\Exp[ M_{\nu_2}  \wedge m | \mathcal{F}_{\nu_1} ] \le M_{\nu_1}  \wedge m$. An upward limit $m \to \infty$ with Proposition~\ref{prop:cond-beppo-levi} concludes the proof.
\end{proof}


\subsubsection{Connection to the local (super)martingale property}

Local martingales form a process class of particular interest as all It\^o diffusions with well-behaved integrands are local martingales \citep[Corollary 25.19]{klenke2013probability}. Still working in discrete times for now, recall that a process $
M$ is said to satisfy a property $\pi$ \emph{locally} if there exists an increasing sequence of stopping times $\{ \tau_j \}$ such that $\lim_{j \to \infty} \tau_j = \infty$ a.s., and for each $j$ the stopped process $ M^ {\tau_j}$ satisfies $\pi$. For ENSMs, we have the following statements.

\begin{proposition}\label{prop:local}
  Local ENSMs are ENSMs. 
\end{proposition}
We remark that local classical NSMs are actually classical NSMs (and thus ENSMs), because local supermartingales bounded from below must be supermartingales (see e.g.\ \cite[p.74]{kallianpur2014stochastic}).
\begin{proof}[Proof of Proposition~\ref{prop:local}]
     Let $M$ be a local ENSM with localizing sequence $\{ \tau_j \}$. Note that $\lim_{j \to \infty} \\ M_{n\wedge \tau_j} = M_n$ almost surely. Thus local nonnegativity implies nonnegativity --- so $M$ is indeed nonnegative (almost surely).

    Now, since $M^{\tau_j}$ is an ENSM, for any $n, N \ge 1$ and and $A \in \mathcal{F}_{n-1}$,
    \begin{equation}
        \Exp[ \id_{A \cap \{ M_{(n-1) \wedge \tau_j}\le N \}} M_{n \wedge \tau_j}  ] \le   \Exp[ \id_{A \cap \{ M_{(n-1) \wedge \tau_j} \le N \} } M_{(n-1) \wedge \tau_j} ].
    \end{equation}
    Hence,
    \begin{equation}
    \label{eq:intermediate-ineq}
        \liminf_{j\to \infty}  \Exp[ \id_{A \cap \{ M_{(n-1) \wedge \tau_j} \le N \}} M_{n \wedge \tau_j}  ] \le   \liminf_{j\to \infty}  \Exp[ \id_{A \cap \{ M_{(n-1) \wedge \tau_j} \le N \} } M_{(n-1) \wedge \tau_j} ].
    \end{equation}
    The left hand side of~\eqref{eq:intermediate-ineq} can be further bounded by Fatou's lemma as
    \begin{align}
       \liminf_{j\to \infty} \Exp[ \id_{A \cap \{ M_{(n-1) \wedge \tau_j} \le N \}} M_{n \wedge \tau_j}  ] &\ge   \Exp[ \liminf_{j\to \infty}  \id_{A \cap \{ M_{(n-1) \wedge \tau_j} \le N \}} M_{n \wedge \tau_j}  ]  \nonumber \\
       &= \Exp\left[  \id_{A \cap \{ M_{n-1} \le N \}} M_{n}  \right]. 
    \end{align}
    Since $0 \le \id_{A \cap \{ M_{(n-1) \wedge \tau_j} \le N \} } M_{(n-1) \wedge \tau_j} \le N$, we can use Lebesgue's dominated convergence theorem (classical, no $\infty$ involved) to simplify the right hand side of~\eqref{eq:intermediate-ineq}:
    \begin{align}
        \lim_{j\to \infty}  \Exp[ \id_{A \cap \{ M_{(n-1) \wedge \tau_j} \le N \} } M_{(n-1) \wedge \tau_j} ] &=  \Exp[ \lim_{j\to \infty} \id_{A \cap \{ M_{(n-1) \wedge \tau_j} \le N \} } M_{(n-1) \wedge \tau_j} ] \nonumber \\
        &= \Exp[ \id_{A \cap \{ M_{n-1} \le N \} } M_{n-1}  ].
    \end{align}
    Combining the three displayed (in)equalities above, we obtain
    \begin{equation}
        \Exp[ \id_{A \cap \{ M_{n-1} \le N \}} M_{n}  ] \le   \Exp[ \id_{A \cap \{ M_{n-1} \le N \} } M_{n-1 } ].
    \end{equation}
    Letting $N \to \infty$,
    \begin{equation}
        \Exp[ \id_{A \cap \{ M_{n-1} < \infty \}} M_{n}  ] \le   \Exp[ \id_{A \cap \{ M_{n-1} < \infty \} } M_{n-1} ].
    \end{equation}
    Note that we also have
    \begin{equation}
        \Exp[ \id_{A \cap \{ M_{n-1} = \infty \}} M_{n}  ] \le   \Exp[ \id_{A \cap \{ M_{n-1} = \infty \} } M_{n-1} ].
    \end{equation}
    So
    \begin{equation}
        \Exp[ \id_{A} M_{n}  ] \le   \Exp[ \id_{A } M_{n-1 } ].
    \end{equation}
    $M$ is an ENSM follows thus from Proposition~\ref{prop:cond-exp-order-rule}.
\end{proof}

Combining Proposition~\ref{prop:M0+classical-SM} and Proposition~\ref{prop:local}, we immediately have the following.
\begin{proposition}\label{prop:M0+local-SM}
    Let $M_0$ be a $\mathcal{F}_0$-measurable nonnegative random variable and let $ 
Y$ be a classical local supermartingale (not necessarily nonnegative) such that $M := M_0 + Y \ge 0 $. Then, $M$ is an ENSM.
\end{proposition}
\begin{proof}[Proof of Proposition~\ref{prop:M0+local-SM}]
    Let $\{ \tau_j \}$ be the localizing sequence. Since $Y^{ \tau_j} $ is a supermartingale, $M^{ \tau_j}$ is an ENSM due to Proposition~\ref{prop:M0+classical-SM}. So $ M$ is a local ENSM, hence an ENSM due to Proposition~\ref{prop:local}.
\end{proof}

The \emph{raison d'\^etre} of Proposition~\ref{prop:M0+local-SM} is that it implies a very rich class of processes to be ENSMs --- those satisfying a stochastic differential equation (SDE) with nonpositive drift. In the following proposition, $t \ge 0$ denotes continuous time and $n \ge 0$ discrete time.

\begin{proposition}\label{prop:sde}
    Let $\{B_t\}_{t \ge 0}$ be a standard brownian motion adapted to a right-continuous filtration $\{ \mathcal{F}_t\}_{t \ge 0} $. Suppose the continuous-time, finite, nonnegative process $\{ M_t \}$ satisfies
    \begin{equation}
        M_t = M_0 +  \int_{0}^t (\sigma_s \d B_s + b_s \d s),
    \end{equation}
    such that $b_t \le 0$ and $\int_0^t \sigma_s^2 < \infty$ a.s. for all $t \ge 0$. Then,  $\{ M_n \}_{n \ge 0}$ is an ENSM.
\end{proposition}
The proof simply follows from the fact that the It\^o integral $\int_{0}^t \sigma_s \d B_s $ is a local martingale \citep[Corollary 25.19]{klenke2013probability}. We remark that if $M_0$ is integrable, $\{ M_t \}$ would be a supermartingale in the classical sense.
 We believe a suitable modification of the previous framework can establish that $\{ M_t \}_{t \ge 0}$ is a continuous-time ENSM, which, though not difficult, is beyond the scope of our paper as we do not commit ourselves to the study of continuous-time ENSMs.




\subsection{Extended NSMs and solutions to drift-free SDEs}\label{sec:test-sde}

In this section, we provide further details on the link between the proper mixture in Section~\ref{sec:mix-nmu0c} and the improper mixture in Section~\ref{sec:mix-flt}, by examining the classical and extended NSMs in Proposition~\ref{prop:gaussian-mixed} and Proposition~\ref{prop:flat-mix-cs} through the lense of continuous-time stochastic analysis.

We put $\mu_0 = 0$ without loss of generality and consider $\normal{0}{1}$ as the distribution of $X_1, X_2, \dots$. Let $\{B_t\}_{t \ge 0}$ be the standard Brownian motion, so that $\{S_n\}_{n \ge 0}$ is the discretization to $\{B_t\}_{t \ge 0}$. Replacing $n$ with $t$ and $S_n$ with $B_t$, the Gaussian-mixed martingale \eqref{eqn:guassian-mixed-mtg} is
\begin{equation}
    L^{(c)}_t = \sqrt{\frac{c^2}{c^2 +t}} \exp\left( \frac{B_t^2}{2(c^2 + t)} \right) =: f(t, B_t).
\end{equation}
We have the partial derivatives
\begin{equation}\label{eqn:f-pd}
\begin{cases}
    f_1(t, B_t) = -f(t, B_t) \frac{c^2 + t +B_t^2}{2(c^2+t)^2},
    \\
    f_2(t,B_t) = f(t, B_t) \frac{B_t}{c^2 + t},
    \\
    f_{22}(t,B_t)= f(t, B_t) \frac{c^2 + t +B_t^2}{(c^2+t)^2}.
    \end{cases}
\end{equation}
By It\^o's lemma,
\begin{equation}\label{eqn:L-sde}
    \d L_t^c =\left( f_1 + \frac{1}{2}f_{22} \right) \d t + f_2 \,\d B_t = 0\, \d t + \frac{L_t^c B_t}{c^2 + t} \,\d B_t.
\end{equation}
This SDE, together with the initial condition $L_0^c = 1$, characterizes the evolution of this test martingale.

Switching to the flat-mixed case, the ENSM \eqref{eqn:flat-mixed-ensm} is
\begin{equation}
    K_t = t^{-1/2} \exp\left( \frac{B_t^2}{2t} \right) =: g(t, B_t).
\end{equation}
We have
\begin{equation}\label{eqn:g-pd}
\begin{cases}
    g_1(t, B_t) = -g(t, B_t) \frac{t +B_t^2}{2t^2},
    \\
    g_2(t, B_t) = g(t, B_t) \frac{B_t}{t},
    \\
    g_{22}(t, B_t) =g(t, B_t) \frac{t +B_t^2}{t^2}.
    \end{cases}
\end{equation}
Again, It\^o's lemma implies
\begin{equation}\label{eqn:K-sde}
\d K_t = \left( g_1 + \frac{1}{2}g_{22} \right) \d t + g_2 \,\d B_t = 0 \,\d t + \frac{K_t B_t}{t} \,\d B_t . 
\end{equation}
The evolution of $\{K_t\}$ is given by the above SDE. Note that $K_0 = \infty$. We can say that for any $\epsilon > 0$, $\{K_t\}_{t \ge \epsilon}$ is the strong solution to the SDE with initial condition $K_\epsilon$. Comparing the partial derivatives in \eqref{eqn:f-pd} and \eqref{eqn:g-pd}, as well as the SDEs in \eqref{eqn:L-sde} and \eqref{eqn:K-sde}, we can see that the flat mixture $\{K_t\}$ is the $c\to 0$ limit case of the Gaussian mixture $\{ L_t^{(c)} \}$ in terms of their evolution.

\section{More ways to construct classical NSMs from ENSMs}\label{sec:ensm-to-nsm}

We mentioned in Proposition~\ref{prop:concave-ensm} that upper thresholding an ENSM can immediately yield a classical NSM. In this section, we demonstrate two intuitive ways to derive a classical NSM from an ENSM. First, if we have already observed some early finite values of an ENSM, the rest of the ENSM shall behave like a classical NSM. This is formalized as the regular conditional distribution. Let $(\mathbb E, \mathcal{E})$ be the Borel space of $\overline{\mathbb R^+}$-valued sequences.

\begin{proposition}[Conditioning]\label{prop:cond-ensm} Let $ M$ be an ENSM adapted to $\mathcal{F}$. Let $\kappa: \Omega \times \mathcal{E} \to \mathbb [0,1]$ be 
a regular conditional distribution of $ M$ given $\mathcal{F}_0$. Then, the distribution $\kappa(\omega, \cdot )$ defines a classical NSM 
with respect to its own filtration
for $\Pr$-almost all $\omega \in \{ M_0 < \infty \}$.
\end{proposition}

In particular, if we are to condition on the event $\{ M_0 = m_0 \}$ where $m_0 < \infty$ is in the range of $M_0$, since each $\kappa(\cdot, B)$ is $\sigma(M_0)$-measurable, there is a well-defined $\kappa( M_0^{-1}(m_0), B )$ by the factorization lemma and $\kappa( M_0^{-1}(m_0), \cdot )$ is a distribution of classical NSM.

\begin{proof}[Proof of Proposition~\ref{prop:cond-ensm}]
Recall the following property of regular conditional distribution \citep[Theorem 8.38]{klenke2013probability}: let $f: \mathbb E \to \mathbb R$ be a measurable functional of processes, then, for $\Pr$-almost all $\omega$,
\begin{equation}\label{eqn:reg-cond-exp}
    \Exp[ f(M_0, M_1,\dots) | \mathcal{F}_0 ](\omega) = \int f(m) \kappa(\omega, \d m),
\end{equation}
as long as $f(M_0, M_1,\dots)$ is integrable.
It is not hard to see the above still holds if we assume $f: \mathbb E \to \overline{\mathbb R^+}$ without integrability instead.

Fix some $n\ge1$, $A$ a Borel set in $\mathbb R^n$ and let us take the following functionals: $$f_1 (m) =  
m_n \id_{ (m_0, \dots, m_{n-1}) \in A } $$ and $$f_2(m) = m_{n-1} \id_{ (m_0, \dots, m_{n-1}) \in A }.$$ Since $\{ 
(M_0, \dots, M_{n-1}) \in A \} \in \mathcal{F}_{n-1}$, we have 
$$\Exp[ f_1(M_0, M_1,\dots) | \mathcal{F}_0 ](\omega) \le \Exp[ f_2(M_0, M_1,\dots) | \mathcal{F}_0 ](\omega).$$ Hence $\int f_1(m) \kappa(\omega, \d m) \le \int f_2(m) \kappa(\omega, \d m)$. This implies that a process distributed according to $\kappa(\omega, \cdot)$ is an ENSM.

Finally, we take $f(m) = m_0$. If $M_0(\omega) < \infty$, the right hand side of \eqref{eqn:reg-cond-exp}, which is $\int f(m)  \kappa(\omega, \d m) < \infty$, implies that the expected value of the 0-th random variable of the process distributed according to $\kappa(\omega, \cdot)$ is finite. This concludes the proof by calling Proposition~\ref{prop:finite-moment-ensm-is-nsm}.
\end{proof}

The seemingly abstract Proposition~\ref{prop:cond-ensm} is straightforward to apply to specific processes. In \cref{ex:cauchy}, conditioning on $M_0 = 1$, for example, leads to a process $ M' $ having exactly the same iteration $M'_n = M_{n-1}'(1/2  + \xi_n)$, but starting with a deterministic initial value $M'_0 = 1$.

Another simple idea is \emph{dividing} the ENSM $M$ by $M_0$. Although $M_0$ may take values in $\overline{\mathbb R^+}$, the following proposition only involves division by positive real number.

\begin{proposition}[Division]\label{prop:div-ensm}
Let $M$ be an EN(S)M adapted to $\mathcal{F}$. The process
\begin{equation}
     N_n = \id_{\{ 0 < M_0 < \infty \}} \frac{M_n}{M_0} .
\end{equation}
is a classical N(S)M.
\end{proposition}

\begin{proof}[Proof of Proposition~\ref{prop:div-ensm}] If $M$ is an ENSM, for any $n \ge 1$,
    \begin{align} 
        \Exp[ N_n | \mathcal{F}_{n-1} ] =  \frac{\id_{\{ 
0 < M_0 < \infty \}}}{M_0}  \Exp[ M_n | \mathcal{F}_{n-1} ]  \le \frac{\id_{\{ 
0 < M_0 < \infty \}}}{M_0}  M_{n-1} = N_{n-1}.
    \end{align}
    Since $N_0 \le 1$, combining above with Proposition~\ref{prop:finite-moment-ensm-is-nsm} concludes the proof. For the ENM case, replace the $\le$ above with an equality.
\end{proof}

\section{More on confidence sequences for the subGaussian mean}\label{sec:cs-appendix}

\subsection{Mixing distribution $\normal{\eta}{1/c^2}$}\label{sec:mix-n0c}

We now present a sequential test and confidence sequence with a mixing distribution that incorporates the experimenter's ``guess'' of the true mean $\mu_0$. The idea is simply that one replaces the mean of the mixing distribution, which in \cref{sec:mix-nmu0c} is the same as the null $\mu_0$, with one's guess $\eta$. In other word, the prior is fixed, $\mu \sim \normal{0}{1/c^2}$, when facing different nulls. We have the following.

\begin{proposition}\label{prop:mix-n0c}
    For any $c>0$, the process
    \begin{equation}
        L_n^c  =  \sqrt{\frac{c^2 }{c^2 + n}} \exp\left( \frac{1}{2}n (\mu_0 - \avgX{n})^2 - \frac{c^2}{2(c^2 + n)n} S_n^2   \right)
    \end{equation}
    is a $\subgaussian{\mu_0}{1}$-supermartingale and consequently, the intervals 
    \begin{equation}\label{eqn:gaussian-prior-known-var-cs}
          \left[ \avgX{n} \pm \sqrt{ \frac{ \log \frac{c^2 + n}{c^2 \alpha^2}  + \frac{c^2 n \avgX{n}^2}{(c^2+n)} }{n} } \right].
    \end{equation}
    form a $(1-\alpha)$-CS for $\mu_0$ over $\subgaussian{\mu_0}{1}$.
\end{proposition}

\begin{proof}[Proof of Proposition~\ref{prop:mix-n0c}] By direct calculation, the mixed martingale is
\begin{align}
   L_n^c = & \int_\mu \ell^\mu_n \, \d \normal{0}{c^{-2}}(\mu)
    =  \frac{\int_{\mathbb R} \exp(-\frac{1}{2}\sum (X_i - \mu)^2 ) \frac{c}{\sqrt{2\pi}} \exp(-c^2 \mu^2/2)  \d \mu }{\exp(-\frac{1}{2}\sum (X_i - \mu_0)^2 )}
   \\
   = &  \frac{
   \sqrt{\frac{c^2 }{c^2 + n}} \exp( S_n^2/2(c^2 + n) - V_n/2 ) }{\exp( -\frac{1}{2}n \mu_0^2 + S_n \mu_0 - V_n/2 )}
   \\
   = & \sqrt{\frac{c^2 }{c^2 + n}} \exp\left( \frac{1}{2}n (\mu_0 - \avgX{n})^2 - \frac{c^2}{2(c^2 + n)n} S_n^2   \right).
\end{align}
This is a nonnegative martingale, issued at $L_0^c = 1$. Applying Ville's inequality, with probability at least $1 - \alpha$ for all $n$,
\begin{equation}
    \frac{1}{2}n (\mu_0 - \avgX{n})^2 - \frac{c^2}{2(c^2 + n)n} S_n^2  \le \log (1/\alpha) - \frac{1}{2} \log \frac{c^2}{c^2+n},
\end{equation}
whence we get a CS
\begin{equation}
    \left[ \avgX{n} \pm \sqrt{ \frac{ \log \frac{c^2 + n}{c^2 \alpha^2}  + \frac{c^2 n \avgX{n}^2}{(c^2+n)} }{n} } \right],
\end{equation}
as claimed.
\end{proof}

We see that the CS \eqref{eqn:gaussian-prior-known-var-cs} becomes wider if $\avgX{n}$ is far away from 0, which would happen if $\mu_0$, the parameter of interest, is far from zero. The effect is clearer if a $N_{\eta,1/c^2}$ prior is used instead, in which case the CS becomes
\begin{equation}\label{eqn:gaussian-prior-known-var-cs-m}
    \left[ \avgX{n} \pm \sqrt{ \frac{ \log \frac{c^2 + n}{c^2 \alpha^2}  + \frac{c^2 n (\avgX{n}-\eta)^2}{(c^2+n)} }{n} } \right].
\end{equation}
Here, the CS \eqref{eqn:gaussian-prior-known-var-cs-m} becomes wider if $\mu_0$, and hence $\avgX{n}$, is far away from $\eta$. We see that the mean of the mixing distribution, $\eta$, is the experimenter's ``guess'' of $\mu_0$. The closer the guess is, the tighter the CS becomes.


Comparing \eqref{eqn:gaussian-prior-known-var-cs-m} to the confidence sequence \eqref{eqn:gaussian-prior-known-var-cs-correct}, we see that ``guessing'' the mean $\mu_0$ with a prior mean $\eta$ can be better than not guessing (setting the prior mean to the true mean $\mu_0$) if the guess is close to the \emph{empirical} mean,
\begin{equation}
    (\avgX{n} - \eta)^2 \le \frac{c^2 + n}{n^2} \log \frac{c^2 + n}{c^2 \alpha^2}.
\end{equation}
In particular, if one knows the true mean beforehand, it is easy to check that the above is satisfied by $\eta = \mu_0$ with high probability as this is just the confidence sequence \eqref{eqn:gaussian-prior-known-var-cs-correct}.

\subsection{Sequential test and CS by conditioning}  \label{sec:mix-flt-cond}
Recall that
\begin{equation}\label{eqn:flatensm}
  K_n = \frac{1}{\sqrt{n}} \exp\left( \frac{1}{2}n (\mu_0 - \avgX{n})^2 \right)
\end{equation}
is the flat mixed likelihood ratio and has infinite mean. Instead of applying the extended Ville's inequality, we may turn it into a classical nonnegative supermartingale by \emph{conditioning} it on $X_1 = \mu_0$ and obtain an integrable NM. While we have introduced a general conditioning property in Proposition~\ref{prop:cond-ensm}, let us go into the details for our particular subGaussian application. Let $n \ge 1$, we define $K_{n-1}'$ by replacing the first observation $X_1$ with $\mu_0$ in $K_n$,
\begin{equation}
 \footnotesize  K_{n-1}' = \frac{1}{\sqrt{n}} \exp\left( \frac{1}{2}n \left(\mu_0 - \frac{\mu_0 + X_2 + \dots + X_n}{n} \right)^2 \right) = \frac{1}{\sqrt{n}} \exp\left( \frac{((n-1)\mu_0 - (S_n - S_1))^2}{2n} \right).
\end{equation}
Note that $\{S_n - S_1\}_{n \ge 1}$ has the same (joint) distribution as $\{ S_{n-1} \}_{n \ge 1}$. So, if we define
\begin{equation}
    K_{n-1}'' = \frac{1}{\sqrt{n}} \exp\left( \frac{((n-1)\mu_0 - S_{n-1})^2}{2n} \right),
\end{equation}
then the process $ K'' $ has the same finite-dimensional distributions as $K'$.

One may generalize the above by conditioning on $X_1 = \dots = X_\nu = \mu_0$, getting the conditioned process
\begin{equation}
  \footnotesize   K_{n-\nu}^{(\nu)} = \frac{1}{\sqrt{n}} \exp\left( \frac{1}{2}n \left(\mu_0 - \frac{\nu\mu_0 + X_{\nu + 1} + \dots + X_n}{n} \right)^2 \right) = \frac{1}{\sqrt{n}} \exp\left( \frac{((n-\nu)\mu_0 - (S_n - S_{\nu}))^2}{2n} \right),
\end{equation}
and then shifting it to get an identically distributed process
\begin{equation}
    K^{[\nu]}_{n-\nu} = \frac{1}{\sqrt{n}}\exp\left( \frac{((n-\nu)\mu_0 - S_{n-\nu})^2}{2n} \right);
\end{equation}
or equivalently written,
\begin{equation}\label{eqn:conditioned-nm-def}
     K^{[\nu]}_{n} = \frac{1}{\sqrt{n+\nu}}\exp\left( \frac{(n\mu_0 - S_{n})^2}{2(n+\nu)} \right).
\end{equation}
For $\nu > 0$, it turns out that $K^{[\nu]}$ is nothing else but the $\normal{\mu_0}{\nu}$-mixed martingale $L^{(\sqrt{\nu})}$ \eqref{eqn:guassian-mixed-mtg}, multiplied by a constant of $\nu^{-1/2}$; less surprisingly, if we allow $\nu$ to take 0, $ K^{[0]} $ is just the flat-mixed ENSM $K$ \eqref{eqn:flatensm}. Therefore, the processes derived above via flat prior and conditioning is equivalent to putting a Gaussian prior described in \cref{sec:mix-nmu0c}. The number of hypothetical observations of value $\mu_0$ we condition on, $\nu$, plays the same role as $c^2$, the prior precision of the Gaussian mixture. This pleasantly agrees with the Bayesian update rule from the improper flat prior. To wit, if we observe $\nu$ data points with value $\mu_0$, the flat prior will be updated to $\pi(\mu | \text{data}) \propto [\exp\left( -\frac{1}{2}(\mu - \mu_0)^2 \right)]^\nu$, which is indeed $\normal{\mu_0}{1/\nu}$ (more on this in \cref{sec:bayes}). Consequently, a CS identical to \eqref{eqn:gaussian-prior-known-var-cs-correct} is obtained.

\begin{proposition}
     For any integer $\nu \ge 1$, define the process $K^{[\nu]} $ as in \eqref{eqn:conditioned-nm-def}. Then, it is a $\subgaussian{\mu_0}{1}$-supermartingale. Consequently, we have the following confidence sequence for the mean $\mu_0$ over $\subgaussian{\mu_0}{1}$,
    \begin{equation}
        \left[ \avgX{n} \pm \frac{1}{n} \sqrt{(n+\nu) \log\frac{n+\nu}{\alpha^2 \nu} } \right],
    \end{equation}
    equivalent to mixing $\mu \sim \normal{\mu_0}{\nu}$ in \cref{prop:gaussian-mixed}.
\end{proposition}


\subsection{Sequential test and CS by the division process}\label{sec:div}

Apart from conditioning on $X_1 = \cdots = X_\nu = \mu_0$, a related and similar idea to get around the infinite expectation issue of $ K^1$ is to \emph{divide} it with $K^1_\nu$. We have formally stated Proposition~\ref{prop:div-ensm} which produces the following classical test martingale and confidence sequence.

\begin{proposition}\label{prop:div}
     Let integers $\nu \ge 1$ and $n \ge 0$. Denote by $\avgX{\nu:n+\nu}$ the sample average from the $(\nu+1)$-st to the $(n + \nu)$-th observations. Then, the process
     \begin{equation}
         O_n^{(\nu)} = \frac{\sqrt \nu}{\sqrt{n+\nu}} \exp \left( \frac{n}{2}(\mu_0 - \avgX{\nu:n+\nu})^2 - (S_{n+\nu} - S_\nu)^2/2n   + S_{n+\nu}^2/2(n+\nu) - S_\nu^2/2\nu \right)
     \end{equation}
     is a $\subgaussian{\mu_0}{1}$-supermartingale and consequently, the intervals
     \begin{equation}
  \left[ \avgX{\nu:n+\nu} \pm \sqrt{\frac{\log \frac{n + \nu}{\alpha^2 \nu } + (S_{n+\nu} - S_\nu)^2/n   - S_{n+\nu}^2/(n+\nu) + S_\nu^2/\nu}{n}}  \right]
\end{equation}
     form a $(1-\alpha)$-CS for $\mu_0$ over $\subgaussian{\mu_0}{1}$.
\end{proposition}

Interesting enough, the test martingale and CS above do not explicitly require our extended theory of supermartingales as one may verify via classical probability theory that $O^{(v)}$ is a martingale. However, it is difficult to \emph{derive} them without the insight of our inventions.

\begin{proof}[Proof of Proposition~\ref{prop:div}]
Note that
        \begin{align}
              & \frac{K_{n+\nu}^1}{K_\nu^1} = \frac{\sqrt \nu}{\sqrt{n+\nu}} \exp \left( \frac{1}{2} n \mu_0^2 - (S_{n+\nu} - S_\nu) \mu_0 + S_{n+\nu}^2/2(n+\nu) - S_\nu^2/2\nu \right)
            \\
            = & \frac{\sqrt \nu}{\sqrt{n+\nu}} \exp \left( \frac{n}{2}(\mu_0 - \avgX{\nu:n+\nu})^2 - (S_{n+\nu} - S_\nu)^2/2n   + S_{n+\nu}^2/2(n+\nu) - S_\nu^2/2\nu \right) = O_n^{(\nu)}
        \end{align}
which is a NM with expectation 1. On $O^{(\nu)}$ Ville's inequality implies,
\begin{equation}
  \left[ \avgX{\nu:n+\nu} \pm \sqrt{\frac{\log \frac{n + \nu}{\alpha^2 \nu } + (S_{n+\nu} - S_\nu)^2/n   - S_{n+\nu}^2/(n+\nu) + S_\nu^2/\nu}{n}}  \right]
\end{equation}
forms a $(1-\alpha)$-CS for $\mu_0$.
\end{proof}

\subsection{One-sided CS by the half flat prior}\label{sec:one-sided-cs}
The following is a confidence sequence based on \cref{prop:flat-mix-cs-one-sided}.
\begin{corollary}\label{cor:one-sided-CS}
    Letting $c_\alpha$ be the solution (which must uniquely exist) in $c\in(0,1)$ to the equation $3c + 2 c \log (2/1.79c) = \alpha$, the
following is a $(1-\alpha)$-CS for $\mu_0$ over $\subgaussian{\mu_0}{1}$:
\begin{equation}\label{eqn:flat-prior-one-sided-cs}
    \left[ \avgX{n} - \frac{V^{-1}\left( \frac{\sqrt{4n}}{c_\alpha} \right)}{\sqrt{n}}, \infty \right).
\end{equation}
\end{corollary}
\begin{proof}[Proof of Corollary~\ref{cor:one-sided-CS}]
    Recall that $Q_1 \sim \frac{1}{2} V(Z)$ where $Z = X_1 - \mu_0$ is a centered 1-subGaussian random variable. Let $c < 1/2$. We need to upper bound
\begin{equation}
    \Pr[ Q_1 \ge 1/c ] = \Pr[ V(Z) \ge 2/c ] \quad \text{and} \quad \Exp[  \id_{\{ Q_1 < 1/c \}} Q_1 ] = \frac{1}{2}\Exp[  \id_{\{ V(Z) < 2/c \}} V(Z) ].
\end{equation}
Let us use the following simple bounds,
\begin{gather}
\forall x, \; V(x)  < 2\exp(x^2/2);
\\
\forall x, \; V(x) \ge 4 \implies  x > 1.26 \implies V(x) >  1.79 \cdot \exp(x^2/2).
\end{gather}
Note that the above implies
\begin{gather}
    \{ V(Z) \ge 2/c \} \subseteq  \{ 2\exp(Z^2/2) \ge 2/c \} ,
    \\
    \{ V(Z) < 2/c \} \subseteq \{ 1.79 \cdot \exp(Z^2/2) < 2/c \}
\end{gather}

We thus have
\begin{align}
     \Pr[ Q_1 \ge 1/c ] = \Pr[ V(Z) \ge 2/c ] \le \Pr[ 2\exp(Z^2/2) \ge 2/c ] =  \Pr[ K_1 \ge 1/c ] \le 2c,
\end{align}
where the last inequality is simply a redux of \eqref{eqn:subgaussian-ClogCsq-tail}. We thus also have
\begin{align}
    & \Exp[  \id_{\{ Q_1 < 1/c \}} Q_1 ] = \frac{1}{2}\Exp[  \id_{\{ V(Z) < 2/c \}} V(Z) ] \le \frac{1}{2}\Exp[  \id_{\{ 1.79 \exp(Z^2/2) < 2/c \}} 2\exp(Z^2/2) ]
    \\
    = & \Exp[  \id_{\{ \exp(Z^2/2) < 2/1.79c \}} \exp(Z^2/2) ]  =  \Exp[  \id_{\{ |Z| < \sqrt{2\log (2/1.79c)} \}} \exp(Z^2/2) ] 
 \le 1 + 2\log (2/1.79c),
\end{align}
where the last inequality is due to \cref{lem:subgaussian-psi2-bound}. We are now ready to use extended Ville's inequality on $\{ Q_n \}_{n \ge 1}$, which gives,
\begin{equation}
    \Pr[ \exists n \ge 1, \ Q_n \ge 1/c ] \le 2c + c(1 + 2\log (2/1.79c)) = 3c + 2 c \log  (2/1.79c).
\end{equation}
$c_\alpha$ is the unique $c \in (0, 1/2)$ such that the RHS of above is $\alpha \in (0,1)$. We have,
\begin{equation}
    \Pr\left[ \forall n\ge 1,\ \avgX{n} - \mu_0 \le \frac{V^{-1}\left( \frac{\sqrt{4n}}{c_\alpha} \right)}{\sqrt{n}}
    \right] \ge 1-\alpha,
\end{equation}
as claimed.
\end{proof}


\section{Bayesian perspectives}
\label{sec:bayes}

In this section, we present some links between the method of mixtures (proper or improper) of likelihood ratios in \cref{sec:estimation} and classical Bayesian methods.

\subsection{Likelihood ratios and Bayes factors}\label{sec:BF}

Let $\{ P_{\theta} : \theta \in \Theta \}$ be a parametric family of distributions on $\mathbb R$, each $P_{\theta}$ having a density $p_{\theta}$. The likelihood ratio martingale with null $\theta_0$ and alternative $\theta_1$ we have considered in \cref{sec:subG-conf-lr} is
\begin{equation}
    \ell_n = \frac{ \prod_{i=1}^n p_{\theta_1}(X_i) }{ \prod_{i=1}^n p_{\theta_0}(X_i) }.
\end{equation}
Let $\pi$ be a mixing measure on $\mathbb R$ that can either be proper ($\pi(\mathbb R) = 1$) or improper (infinite). Let us denote the mixed $P_{\theta_0}$-NSM or ENSM by
\begin{equation}\label{eqn:gen-mixed-nsm-or-ensm}
    M_n^{\pi,\theta_0} =  \frac{\int \prod_{i=1}^n p_{\theta_1}(X_i) \pi(\d\theta_1) }{ \prod_{i=1}^n p_{\theta_0}(X_i) }.
\end{equation}

Switching to a Bayesian perspective, 
 let $\pi_0$ be a prior measure on $\mathbb R$ that can either be proper ($\pi_0(\mathbb R) = 1$) or improper (infinite) over the parameter $\theta$. Upon observing the data $X_1, \dots, X_n$, the prior $\pi_0$ evolves into the posterior $ \pi_n$  that satisfies
\begin{equation}\label{eqn:posterior}
 \frac{\d  \pi_n}{\d \pi_0}(\theta) = \frac{ \prod_{i=1}^n p_{\theta}(X_i) }{  \int \prod_{i=1}^n p_{\theta'}(X_i) \pi_0(\d \theta') }.
\end{equation}

Several observations can be made. First, while $\pi_0$ can be any $\sigma$-finite measure, possibly improper, the posterior $\pi_n$ defined above is always a probability measure (i.e.\ $ \pi_n(\mathbb R) = 1$) if the normalization constant in the denominator is finite. Second, the posterior $\pi_n$ is invariant under scalar multiplication of the prior $\pi_0$ (e.g.\ if $\pi_0$ is the flat prior, the thickness does not matter; see \cref{tab:bayes-comp-change}).
Third, flipping \eqref{eqn:posterior} upside down, $ \frac{\d \pi_0}{\d \pi_n}(\theta)$ is exactly the mixture (E)NSM (depending on the finiteness of $\pi_0$) of the likelihood ratio \eqref{eqn:gen-mixed-nsm-or-ensm} with null $\theta$, mixed alternatives $\theta' \sim \pi_0$, i.e.\ $M_n^{\pi_0,\theta}$. This is a fact recorded by, for example, \cite{waudby2020confidence} as the ``prior-posterior-ratio martingales''. When $\pi_0$ is proper, the ratio $ \frac{\d  \pi_n}{\d \pi_0}(\theta_0)$ is commonly referred to as the \emph{Bayes factor} between the point null $\theta_0$ and alternative under the prior $\theta_1 \sim \pi_0$ (the unmixed ratio $1/\ell_n$ being called the Bayes factor between the point null $\theta_0$ and the point alternative $\theta_1$).

\begin{table}[!h]
\footnotesize
    \centering
    \begin{tabular}{p{2.2cm}|p{2.5cm}|c|p{2.2cm}|p{3cm}}
    \hline
         & Likelihood ratio (an ENM) & Bayes factor & Bayes posterior & Confidence sequence \\
         \hline
        {Point null $\theta_0$, point alternative} $\theta_1$ & $ \ell_n = \frac{ \prod_{i=1}^n p_{\theta_1}(X_i) }{ \prod_{i=1}^n p_{\theta_0}(X_i) }$ & $b_{01} = \ell_n^{-1}$ & point mass on $\theta_1$ & $\{ \theta_0 : \ell_n \le 1/\alpha \}$
        \\
        \hline
        {Proper prior $\pi_0$ over alternative $\theta_1$}& $L_n =  \int \ell_n \pi_0(\d \theta_1)$ & $ B_{01} = L_n^{-1}$ & $\d \pi_n= B_{01} \d \pi_0$ & $\{ \theta_0: L_n \le 1/\alpha \}$ \\
        \hline
        {Improper infinite prior $\pi_0'$ over alternative $\theta_1$} &$L_n' =  \int \ell_n \pi_0'(\d \theta_1)$ & $B_{01}' = (L_n')^{-1}$ & $\d \pi_n' = B_{01} \d \pi_0'$ & $\{ \theta_0: L_n' \le f^{-1}_1(\alpha) \}$ \\
        \hline
        {Improper infinite prior $D \cdot \pi_0'$ over alternative $\theta_1$} & $D \cdot L_n'$ & $D^{-1}  B_{01}'$ & $D^{-1}B_{01}' \d (D \pi_0') $$ =  B_{01}' \d \pi_0' = \d \pi_n' $ & $\{ \theta_0: D L_n' \le f^{-1}_D (\alpha)  \}$ $=\{ \theta_0: L_n' \le f^{-1}_1(\alpha) \}$ \\
        \hline
    \end{tabular}
    \caption{Comparison of statistical quantities of interest under prior scalar change. $f_D$ denotes the function $f_D(x) =  x^{-1} \Exp[ DL_0' \id_{\{ D L_0' < x \}} ] + \Pr[ D L_0' \ge x ]$, and it satisfies $f_1(x/D) = f_D(x)$, so consequently $D f_1^{-1}(\alpha) = f_D^{-1}(\alpha)$. Likelihood ratios and Bayes factors are variant under constant thickening of the priors, but posteriors and CSs are not as they are ``compared against'' priors.}
    \label{tab:bayes-comp-change}
\end{table}

\subsection{Flat priors versus increasingly diffuse priors}

Flat priors (as well as other improper priors) are attractive options as objective, non-informative, ``default'' priors. Yet flat priors are typically avoided by Bayesians who use Bayes factors, since the latter is better understood as a comparison between competing hypotheses, requiring both sides of the fraction to be on the same scale. Any mixture that does not properly integrate to 1 on the alternative side will make the comparison unfair \citep{berger1998bayes}. 
Indeed, it can be seen from \cref{tab:bayes-comp-change} that Bayes factors are unfavorably variant under constant changes of improper priors.
Some Bayesians (e.g.\ \cite{berger1996intrinsic}) who nonetheless prefer the objectiveness of improper priors
have proposed a scheme of first learning a small subset of the sample to evolve the improper into a proper, data-dependent prior, and then obtaining the ``fair'' Bayes factor as usual. This can be likened to our note in \cref{sec:mix-flt-cond} that conditioning the ENSM on some particular observations leads to an NSM issued at 1.

To make compatible non-informative priors and Bayes factors, Bayesians have also considered proper Gaussian priors that are \emph{diffuse}, i.e.\ of large variances. \cite{rouder2009bayesian} warn that very diffuse priors cause powerless statistical procedures (e.g.\ Bayes factors) that overly prefer the null. Here, we shall highlight the links and differences between diffuse proper priors and flat improper priors: the flat prior is the limit of increasingly diffuse Gaussian priors \emph{only in terms of posterior}, but not in terms of the Bayes factors, confidence sequences and sequential tests. In particular, proper Gaussian priors have decreasing power when increasingly diffuse, whereas the improper Gaussian prior is very powerful, as justified below.

When testing if the mean equals $\mu_0$ with unit-variance Gaussian data, if we set the flat prior to thickness $D=1$, then the ENSM $\sqrt{2\pi} K $ satisfies
\begin{equation}
    \frac{1}{\sqrt{2\pi} K_n} = p_{\avgX{n}, 1/n}(\mu_0)
\end{equation}
(see the expression of $K_n$ in \cref{prop:flat-mix-cs}), where as a consequence of the observations in \cref{sec:BF}, $ p_{\avgX{n}, 1/n}$ is the density of the Bayesian posterior arising from a flat prior, which is exactly $\normal{\avgX{n}}{1/n}$. 

With Gaussian prior $\pi_0 = \normal{\eta}{c^{-2}}$, we know from the conjugate prior formula that the posterior is 
\begin{equation}
  \pi_n =    \normal{\frac{c^2 \eta + S_n}{c^2 + n}}{ (c^2 + n)^{-1} }.
\end{equation}
Hence, the Bayes factor, which is the inverse of the $\pi_0$-mixed NM $L^{\eta, c}$, satisfies
\begin{multline}\label{eqn:BF-mu-indep}
 B_{01} =   \frac{1}{   L_n^{\eta, c} } = \frac{ p_{\frac{c^2 \eta + S_n}{c^2 + n} ,  (c^2 + n)^{-1}  } (\mu_0)
    }{p_{\eta, c^{-2}} (\mu_0) }  \\  = \sqrt{\frac{c^2 + n}{c^2}} \exp\left( \frac{ c^2 (\mu_0 -\eta)^2 - (c^2+n)(\mu_0 - \frac{c^2 \eta + S_n}{c^2 + n} )^2 }{2} \right).
\end{multline}
In particular, one may verify that setting $\eta$ to 0 above, the $\normal{0}{c^{-2}}$-mixture $L_n^c$ in \cref{prop:mix-n0c} is recovered.

 Fixing the null $\mu_0$, prior mean $\eta$, and observations, when the precision $c^2$ of the prior $\pi_0 = \normal{\eta}{c^{-2}}$ goes to 0,
    \begin{enumerate}
        \item the posterior $\pi_n = \normal{\frac{c^2 \eta + S_n}{c^2 + n}}{ (c^2 + n)^{-1} }$ converges to $\normal{\avgX{n}}{1/n}$;
        \item the NM $L^{\eta, c} = \frac{ \d \pi_0 }{ \d \pi_n }(\mu_0)$ converges to 0;
        \item equivalently, the Bayes factor $B_{01} = \frac{\d \pi_n}{\d \pi_0}$ diverges to $\infty$;
        \item under a fixed alternative $\mu_1 \neq \mu_0$, the power (for testing if the mean equals $\mu_0$) vanishes to 0;
        \item allowing $\mu_0$ to vary, the CS at time $n$ derived via Ville's inequality (previously presented in \eqref{eqn:gaussian-prior-known-var-cs-m}) 
        \begin{equation}
           \CI_n =  \{ \mu_0 : B_{01} \ge \alpha \} = \left[ \avgX{n} \pm \sqrt{ \frac{ \log \frac{c^2 + n}{c^2 \alpha^2}  + \frac{c^2 n (\avgX{n}-\eta)^2}{(c^2+n)} }{n} } \right]
        \end{equation}
        diverges to $\mathbb R$.
    \end{enumerate}

If we set the prior mean $\eta$ to equal the null $\mu_0$, 
everything under fixed $\mu_0$ (the posterior, NM, Bayes factor, test) equals those under $\normal{\eta}{c^{-2}}$ with $\eta$ replaced by $\mu_0$; only the CS is computed differently since it is solved from varying $\mu_0$. Therefore, when the prior precision $c^2 \to 0$, the posterior, NM,  Bayes factor, and test power have the same limits as mentioned above; the CS, which we have derived in \cref{prop:gaussian-mixed},
\begin{equation}
    \left[ \avgX{n} \pm \sqrt{ \frac{ \left( \log \frac{c^2 + n}{c^2\alpha^2} \right)(1+c^2/n)}{n} } \right]
\end{equation}
diverges to $\mathbb R$ as well given any $n$ observations.

Here we see an interesting difference among statistical quantities under diffuse Gaussian priors. The ``limiting posterior'' (as the prior precision $c^2 \to 0$) is $\normal{\avgX{n}}{1/n}$, which equals the posterior from the flat prior. However, the NMs, Bayes factors, and CSs using Gaussian priors
all become degenerate when increasingly diffuse, not
approximating the non-degenerate ones obtained via the
flat prior. 
Indeed, we have seen earlier that the flat mixture of Gaussian likelihood ratios is not an NM (it is a nonintegrable ENM), it has a finite (though an unfair comparison indeed) Bayes factor; and has non-zero power and a powerful CS (derived using our extended Ville's inequality).

It is perhaps worth remarking, that if we instead apply the \emph{extended} Ville's inequality to the NMs (from $m=1$) obtained from diffuse Gaussian prior mixtures to obtain confidence sequences, they do \emph{not} become degenerate when $c^2 \to 0$ but converge to the powerful flat mixture CS (\cref{cor:Gaussian-CS-flat}) instead, as a consequence of \cref{prop:vil-conv}. However, the exact expressions of these CSs would suffer from an excessive overload of non-closed-form special functions and we do not present them in this paper.

\section{Omitted proofs}\label{sec:pf}

\begin{proof}[Proof of \cref{lem:subgaussian-lr}]
    Let $X_n \sim P$ which is 1-subGaussian with \\ $\Exp[X_n] = \mu_0$. Then,
\begin{align}
    & \Exp[\e^{\frac{(X_n - \mu_0)^2 - (X_n - \mu)^2}{2} }] = \Exp[ \e^{ (\mu - \mu_0)(X_n - \mu_0) - \frac{1}{2}(\mu -\mu_0)^2 } ]
    \\
    = & \Exp[ \e^{ (\mu - \mu_0)(X_n - \mu_0) } ] \, \e^{- \frac{1}{2}(\mu -\mu_0)^2 } 
    \le\e^{ \frac{1}{2}(\mu -\mu_0)^2 }\e^{- \frac{1}{2}(\mu -\mu_0)^2 } = 1.
\end{align}
Hence
    \begin{equation}
       \ell(\mu; \mu_0)_n  =   \exp\left\{ \sum_{i=1}^n  \frac{(X_i - \mu_0)^2 - (X_i - \mu)^2}{2} \right\}
    \end{equation}
    is a supermartingale if $X_1, X_2, \dots$ are indepedent 1-subGaussian. The $\le$ above holds with equality when $X_n \sim \normal{\mu_0}{1}$, in which case it becomes a martingale.
\end{proof}

\begin{proof}[Proof of\cref{prop:gaussian-mixed}] The mixed martingale is
\begin{align}
   L_n^{(c)}  = & \int_{\mathbb R} \ell(\mu;\mu_0)_n \, \d \normal{\mu_0}{c^{-2}} (\mu)
   \\
   = & \frac{\int_{\mathbb R} \exp(-\frac{1}{2}\sum (X_i - \mu)^2 ) \frac{c}{\sqrt{2\pi}} \exp(-c^2 (\mu-\mu_0)^2/2)  \d \mu }{\exp(-\frac{1}{2}\sum (X_i - \mu_0)^2 )}
   \\
   = & \frac{\frac{c}{\sqrt{2\pi}} 
   \int_{\mathbb R} \exp(  -\frac{1}{2}(c^2 + n)\mu^2 + (S_n + c^2 \mu_0 ) \mu -(V_n+c^2 \mu_0^2)/2)  \d \mu }{\exp(-\frac{1}{2}\sum (X_i - \mu_0)^2 )}
   \\
   = &  \frac{
   \sqrt{\frac{c^2 }{c^2 + n}} \exp( (S_n + c^2 \mu_0 )^2/2(c^2 + n) -(V_n+c^2 \mu_0^2)/2 ) }{\exp( -\frac{1}{2}n \mu_0^2 + S_n \mu_0 - V_n/2 )}
   \\
    = &  \frac{
   \sqrt{\frac{c^2 }{c^2 + n}} \exp( \frac{c^4 \mu_0^2 + 2c^2 S_n \mu_0 + S_n^2}{2(c^2 + n)} -(V_n+c^2 \mu_0^2)/2 ) }{\exp( -\frac{1}{2}n \mu_0^2 + S_n \mu_0 - V_n/2 )}
   \\
    = &  \frac{
   \sqrt{\frac{c^2 }{c^2 + n}} \exp(\frac{-c^2 n \mu_0^2 + 2c^2 S_n \mu_0 + S_n^2 - (c^2+n)V_n }{2(c^2 + n)}) }{\exp( -\frac{1}{2}n \mu_0^2 + S_n \mu_0 - V_n/2 )}
   \\
   = &  \sqrt{\frac{c^2 }{c^2 + n}} \exp\left( \frac{n^2 \mu_0^2 - 2 n S_n \mu_0 + S_n^2}{2(c^2 + n)} \right)
   \\
   = & \sqrt{\frac{c^2 }{c^2 + n}}\exp\left( \frac{n^2 (\mu_0 - \avgX{n}) ^2 }{2(c^2 + n)} \right).
\end{align}
This is a nonnegative martingale, issued at $L_0^{(c)} = 1$. Applying Ville's inequality, with probability at least $1 - \alpha$ for all $n$,
\begin{equation}
   \frac{n^2 (\mu_0 - \avgX{n}) ^2 }{2(c^2 + n)}  \le \log (1/\alpha) - \frac{1}{2} \log \frac{c^2}{c^2+n}.
\end{equation}
Rearranging, we get the CS for $\mu_0$:
\begin{equation}
    \left[ \avgX{n} \pm \sqrt{ \frac{ \left( \log \frac{c^2 + n}{c^2\alpha^2} \right)(1+c^2/n)}{n} } \right].
\end{equation}
This completes the proof.
\end{proof}

To prove \cref{prop:flat-mix-cs-one-sided}, 
the following extensions of \cref{lem:subgaussian-lr} are essential (see also \cite[Section 3.2]{ruf2022composite}).

\begin{lemma}\label{lem:lr-mod-onesided}
    For any $\mu, \mu_0, \mu_1, \mu_2,\dots$ and $\mathbbmsl P = P_1\otimes P_2 \otimes \dots $ with $P_i \in \subgaussian{\mu_i}{1}$, the process $\tilde{\ell}_n(\mu;\mu_0;\mathbbmsl P) \newline = \exp \left\{ \sum_{i=1}^n (\mu -\mu_0)(X_i - \mu_i) - \frac{n}{2}(\mu-\mu_0)^2 \right\}$ is a $\mathbbmsl P$-supermartingale.
\end{lemma}
\begin{proof}[Proof of Lemma~\ref{lem:lr-mod-onesided}]
    Let $X_n \sim P_n$ which is 1-subGaussian with mean $\mu_n$. Then,
    \begin{align}
        & \Exp[ \e^{ (\mu - \mu_0)(X_n - \mu_n) - \frac{1}{2}(\mu -\mu_0)^2 } ] =  \Exp[ \e^{ (\mu - \mu_0)(X_n - \mu_n) } ] \e^{ - \frac{1}{2}(\mu -\mu_0)^2}
        \\ \le & \e^{\frac{1}{2}(\mu - \mu_0)^2} \e^{ - \frac{1}{2}(\mu -\mu_0)^2} = 1.
    \end{align}
    So $\{\tilde{\ell}_n(\mu;\mu_0;\mathbbmsl P)\}$ is a $\mathbbmsl P$-supermartingale.
\end{proof}

\begin{lemma}\label{lem:subgaussian-lr-onesided}
    Recall that  $\ell(\mu; \mu_0)_n  =   \exp\left\{ \sum_{i=1}^n  \frac{(X_i - \mu_0)^2 - (X_i - \mu)^2}{2} \right\}$.
    For any $\mu, \mu_0 \in \mathbb R$ such that $\mu \ge \mu_0$, the process $\ell(\mu; \mu_0) $ is a $\subgaussianproc{\le \mu_0}{1}$-supermartingale. However, the process $\ell(\mu; \mu_0) $ is a $\subgaussianproc{\mathsf{ra}\le \mu_0}{1}$-e-process, but it is not a $\subgaussianproc{\mathsf{ra}\le \mu_0}{1}$-supermartingale.
\end{lemma}

\begin{proof}[Proof of Lemma~\ref{lem:subgaussian-lr-onesided}]
     Let $X_n$ be 1-subGaussian with $\Exp[X_n] \le \mu_0$. Then,
\begin{align}
    & \Exp[\e^{\frac{(X_n - \mu_0)^2 - (X_n - \mu)^2}{2}}] = \Exp[ \e^{ (\mu - \mu_0)(X_n - \mu_0) - \frac{1}{2}(\mu -\mu_0)^2 } ]
    \\
    = & \Exp[ \e^{ (\mu - \mu_0)(X_n - \mu_0) } ] \, \e^{- \frac{1}{2}(\mu -\mu_0)^2 } = \Exp[ \e^{ (\mu - \mu_0)(X_n - \Exp[X_n]) } ] \, \e^{ (\mu-\mu_0)( \Exp[X_n] - \mu_0 ) } \e^{- \frac{1}{2}(\mu -\mu_0)^2 }
    \\
    \le & \e^{ \frac{1}{2}(\mu -\mu_0)^2 }\e^{- \frac{1}{2}(\mu -\mu_0)^2 } = 1.
\end{align}
So $ \ell(\mu; \mu_0)$ is a $\subgaussianproc{\le \mu_0}{1}$-supermartingale. Further, for each $\mathbbmsl P =  P_1\otimes P_2 \otimes \dots \in  \subgaussianproc{\mathsf{ra}\le \mu_0}{1} $, say $P_i \in \subgaussian{\mu_i}{1}$, the process $\ell(\mu; \mu_0)$ with $\mu \ge \mu_0$ satisfies
\begin{align}
   & \ell (\mu; \mu_0)_n
    = \exp\left\{ \sum_{i=1}^n (\mu - \mu_0)(X_i - \mu_0) - \frac{n}{2}(\mu -\mu_0)^2 \right\} \le 
    \\ &      \exp\left\{ \sum_{i=1}^n (\mu - \mu_0)(X_i - \mu_i) - \frac{n}{2}(\mu -\mu_0)^2 \right\} = \tilde{\ell}_n(\mu;\mu_0;\mathbbmsl P), \label{eqn:running-nsm}
\end{align}
where the inequality is due to $\mu - \mu_0 \ge 0$ and $\sum_{i=1}^n \mu_i \le n \mu_0$. The process \eqref{eqn:running-nsm} is a $\mathbbmsl P$-supermartingale by Lemma~\ref{lem:lr-mod-onesided}. 

To see that $\ell(\mu; \mu_0) $ is not a $\subgaussianproc{\mathsf{ra}\le \mu_0}{1}$-supermartingale, let $\Delta$ be a real number such that \newline $\e^{-\frac{1}{2}(\mu - \mu_0)^2 + (\mu-\mu_0)\Delta}( \e^{-(\mu-\mu_0)} + \e^{\mu-\mu_0} ) > 2$.
Consider $X_1 \sim \text{rad}_{\mu_0- \Delta }$ and $X_2 \sim \text{rad}_{\mu_0 +\Delta}$. Here $\text{rad}_{\theta}$ denotes the Rademacher distribution on $\{  \theta - 1, \theta + 1 \}$ which is 1-subGaussian. We have 
\begin{equation}
    \Exp[ \e^{ (\mu - \mu_0)(X_2 - \mu_0) - \frac{1}{2}(\mu -\mu_0)^2 } ] = \frac{\e^{-\frac{1}{2}(\mu - \mu_0)^2 + (\mu-\mu_0)\Delta}( \e^{-(\mu-\mu_0)} + \e^{\mu-\mu_0} ) }{2} > 1.
\end{equation}

\end{proof}

\begin{proof}[Proof of \cref{prop:flat-mix-cs-one-sided}]
We calculate the mixture likelihood ratio with the half-of-real-line flat prior $F$ as follows: 
\begin{align}
   & \frac{\int_{\mu_0}^\infty \exp(-\frac{1}{2}\sum (X_i - \mu)^2 )(2\pi)^{-1/2} \d \mu }{\exp(-\frac{1}{2}\sum (X_i - \mu_0)^2 )}
   =  \frac{ (2\pi)^{-1/2} \int_{\mu_0}^\infty\exp( -\frac{1}{2}n \mu^2 + S_n \mu  ) \d \mu }{\exp( -\frac{1}{2}n \mu_0^2 + S_n \mu_0  )}
   \\
   = & \frac{  \sqrt{\frac{1}{4n}} \exp(S_n^2/2n) \left\{1- \erf\left( \sqrt{n/2}(\mu_0 - \avgX{n}) \right) \right\}  }{\exp( -\frac{1}{2}n \mu_0^2 + S_n \mu_0  )}
    \\
    = & \frac{1}{\sqrt{4n}} \exp\left( \frac{1}{2}n (\mu_0 - \avgX{n})^2 \right) \left\{1- \erf\left( \sqrt{n/2}(\mu_0 - \avgX{n}) \right) \right\} 
    \\ = &  \frac 1{\sqrt{4n}} V(\sqrt{n}(\avgX{n} - \mu_0 )) = Q_n.
\end{align}
The desired result follows immediately from Lemma~\ref{lem:subgaussian-lr-onesided} and the closure under mixture property (\cref{cor:mixture-eproc}).
\end{proof}

\end{document}